\documentclass[a4paper,12pt]{article}

\usepackage{subfiles}

\usepackage{amsfonts,amssymb,amsmath,amsthm,latexsym,epsfig,amscd,bbm}
\usepackage[english]{babel}
\usepackage{graphicx}
\usepackage{tikz}
\usepackage{courier}
\usepackage{bbm}
\usepackage{enumerate}
\usepackage{psfrag}
\usepackage{wasysym}
\usepackage{type1cm}

\usepackage[notref,notcite,color]{showkeys}

\def\todo#1{\marginpar{\raggedright \tiny #1}}

\usepackage{calc}

\def\arraypar#1{\parbox[c]{\textwidth - 2cm}{\centering #1}}
\usepackage{environ}
\NewEnviron{display}{
\begin{equation}\begin{array}{c}
\arraypar{\BODY}
\end{array}\end{equation}
}

\def\clap#1{\hbox to 0pt{\hss#1\hss}}

\makeatletter \@addtoreset{equation}{section}
\makeatother

\makeatletter \@addtoreset{enunciato}{section}
\makeatother

\newcounter{enunciato}[section]

\newtheorem{ittheorem}{Theorem}
\newtheorem{itlemma}{Lemma}
\newtheorem{itproposition}{Proposition}
\newtheorem{itdefinition}{Definition}
\newtheorem{itremark}{Remark}
\newtheorem{itclaim}{Claim}
\newtheorem{itfact}{Fact}
\newtheorem{itconjecture}{Conjecture}
\newtheorem{itcorollary}{Corollary}

\newenvironment{theorem}{\addtocounter{enunciato}{1}
\begin{ittheorem}}{\end{ittheorem}}

\newenvironment{lemma}{\addtocounter{enunciato}{1}
\begin{itlemma}}{\end{itlemma}}

\newenvironment{proposition}{\addtocounter{enunciato}{1}
\begin{itproposition}}{\end{itproposition}}

\newenvironment{definition}{\addtocounter{enunciato}{1}
\begin{itdefinition}}{\end{itdefinition}}

\newenvironment{remark}{\addtocounter{enunciato}{1}
\begin{itremark}}{\end{itremark}}

\newenvironment{conjecture}{\addtocounter{enunciato}{1}
\begin{itconjecture}}{\end{itconjecture}}

\newenvironment{corollary}{\addtocounter{enunciato}{1}
\begin{itcorollary}}{\end{itcorollary}}

\newcommand{\be}[1]{\begin{equation}\label{#1}}
\newcommand{\ee}{\end{equation}}

\newcommand{\bl}[1]{\begin{lemma}\label{#1}}
\newcommand{\el}{\end{lemma}}

\newcommand{\br}[1]{\begin{remark}\label{#1}}
\newcommand{\er}{\end{remark}}

\newcommand{\bt}[1]{\begin{theorem}\label{#1}}
\newcommand{\et}{\end{theorem}}

\newcommand{\bd}[1]{\begin{definition}\label{#1}}
\newcommand{\ed}{\end{definition}}

\newcommand{\bp}[1]{\begin{proposition}\label{#1}}
\newcommand{\ep}{\end{proposition}}

\newcommand{\bc}[1]{\begin{corollary}\label{#1}}
\newcommand{\ec}{\end{corollary}}

\newcommand{\bcj}[1]{\begin{conjecture}\label{#1}}
\newcommand{\ecj}{\end{conjecture}}

\newcommand{\bpr}{\begin{proof}}
\newcommand{\epr}{\end{proof}}

\DeclareMathOperator\Cov{Cov}

\def\Z{\mathbb{Z}}
\def\N{\mathbb{N}}
\def\R{\mathbb{R}}

\def\P{\mathbb{P}}
\def\E{\mathbb{E}}

\newcommand{\1}[1]{{\mathbbm{1}}_{#1}}

\def \ba {\begin{array}}
\def \ea {\end{array}}

\def \P  {{\mathbb P}}
\def \E  {{\mathbb E}}

\def \rhoprime {\bar{\rho}}

\def \cA {{\mathcal A}}
\def \cB {{\mathcal B}}
\def \cC {{\mathcal C}}

\def \cH {{\mathcal H}}
\def \cG {{\mathcal G}}
\def \cF {{\mathcal F}}

\def \cW {{\mathcal W}}

\def \cU {{\mathcal U}}
\def \cS {{\mathcal S}}

\usepackage{graphicx}

\DeclareSymbolFont{symbolsC}{U}{pxsyc}{m}{n}
\DeclareMathSymbol{\opentimes}{\mathrel}{symbolsC}{93}

\newcommand{\um}{{\angle}}
\newcommand{\tres}{{\mathbin{\, \text{\rotatebox[origin=c]{180}{$\angle$}}}}}
\newcommand{\treze}{{\mathbin{\text{\rotatebox[origin=c]{35}{$\opentimes$}}}}}

\newcounter{constant}
\newcommand{\newconstant}[1]{\refstepcounter{constant}\label{#1}}
\newcommand{\useconstant}[1]{c_{\textnormal{\tiny \ref{#1}}}}
\usepackage{array}
\newcolumntype{e}{>{\displaystyle}r @{\,} >{\displaystyle}c @{\,} >{\displaystyle}l}
\usepackage[colorlinks]{hyperref}
\usepackage[figure,table]{hypcap} % Correct a problem with hyperref
\hypersetup{final}
\usepackage{xcolor}
\hypersetup{
    colorlinks,
    linkcolor={blue!50!black},
    citecolor={blue!50!black},
    urlcolor={blue!80!black}
}

\usepackage[utf8]{inputenc}

\newcommand{\footremember}[2]{%
    \footnote{#2}
    \newcounter{#1}
    \setcounter{#1}{\value{footnote}}%
}
\newcommand{\footrecall}[1]{%
    \footnotemark[\value{#1}]%
} 

\title{Random walk on random walks: higher dimensions}
\author{
  Oriane Blondel \footnote{CNRS, Univ Lyon, Université Claude Bernard Lyon 1, CNRS UMR 5208, Institut Camille Jordan, 43 boulevard du 11 novembre 1918 -- 69622, France}
  \and
  Marcelo R.\ Hilário \footnote{Universidade Federal de Minas Gerais, Dep. de Matemática, 31270-901 Belo Horizonte} \footremember{NYUSh}{NYU-Shanghai, 1555 Century Av., Pudong Shanghai, CN 200122}
    \and
  Renato S.\ dos Santos \footnote{Weierstrass Institute for Applied Analysis and Stochastics, Mohrenstr. 39, 10117 Berlin}
  \and
  Vladas Sidoravicius \footnote{Courant Institute, NYU, 251 Mercer Street New York, NY 10012} \footrecall{NYUSh}
  \and
  Augusto Teixeira \thanks{Instituto de Matemática Pura e Aplicada, Estrada Dona Castorina 110, 22460-320 Rio de Janeiro}
  }
\date{\today}

\begin{document}

\maketitle

\begin{abstract}
We study the evolution of a random walker on a conservative dynamic random environment composed of independent particles
performing simple symmetric random walks,
generalizing results of \cite{HHSST14}
to higher dimensions and more general transition kernels
without the assumption of uniform ellipticity or nearest-neighbour jumps.
Specifically, we obtain a strong law of large numbers,
a functional central limit theorem and large deviation estimates for the position of the random walker
under the annealed law in a high density regime.
The main obstacle is the intrinsic lack of monotonicity in higher-dimensional, non-nearest neighbour settings.
Here we develop more general renormalization and renewal schemes that allow us to overcome this issue.
As a second application of our methods, we provide an alternative proof of the ballistic behaviour of the front of (the discrete-time version of) the infection model introduced in \cite{KS05}.

\bigskip


\end{abstract}

\vfill

\pagebreak

%%%%\subfile{intromany}
%%%%% Subfile Replaced by the text below %%%%%%%

\section{Introduction}
\label{s:intro}

Random walks on random environments are models for the movement of a tracer particle in a disordered medium,
and have been the subject of intense research for over 40 years.
The seminal works \cite{KKS75, Sinai82, Solomon75},
concerning one-dimensional random walk in static random environment (i.e., constant in time),
established a rich spectrum of asymptotic behaviours that can be very different from that of usual random walks.
In higher dimensions, important questions remain open despite much investigation.
For excellent expositions on this topic, see \cite{Sznitman04, Zeitouni04}.
The \emph{dynamic} version of the model, i.e., when the random environment is allowed to evolve in time,
has been also studied for over three decades (see e.g.\ \cite{BIMP92, Madras86}).
However, models with both space and time correlations have been only considered relatively recently.
For an overview, we refer to the PhD theses \cite{Avenathesis, Santosthesis}.
We will abbreviate ``RWRE'' for random walk in static random environment, 
and ``RWDRE'' for random walk in dynamic random environment.

Asymptotic results for RWDRE under general conditions were derived e.g.\ in \cite{AdHR10, AdHR11, BV16, CDRRS13, dHdSS13, OS16, RV13},
often requiring \emph{uniform mixing} conditions on the random environment 
(implying e.g.\ that the conditional distribution of the environment at the origin
given the initial state uniformly approaches a fixed law for large times).
This uniformity can be relaxed in particular examples, e.g.\ \cite{B16, dHdS14, MV15} (supercritical contact process),
or under additional assumptions, e.g.\ \cite{ABF16a, ABF16b} (spectral gap, weakly non-invariant) and \cite{BH14} (attractivity).
But arguably, some of the most challenging random environments are given by conservative particle systems,
due to their poor mixing properties. 
Such cases have been considered in \cite{AFJV15, AJV14, AdSV13, HS15, dS14} (simple symmetric exclusion), 
and in \cite{HHSST14, dHKS14} (independent random walks).
Each of these works imposes additional conditions and explores very specific properties of the environment in question.
In particular, the works \cite{HHSST14, dHKS14, HS15} introduce \emph{perturbative} approaches,
where parameters of the system are driven to a limiting value where the behaviour is known.

In the present paper, we consider as in \cite{HHSST14} 
dynamic random environments given by systems of independent simple symmetric random walks.
As mentioned above, asymptotic results for this model are challenging since
the random environment is conservative and has slow and non-uniform mixing.
We extend the results of \cite{HHSST14} to higher dimensions and more general transition kernels.
Additional difficulties arise in this setting due to the loss of monotonicity properties 
present in the one-dimensional, nearest-neighbour case.
Our main results are a strong law of large numbers, a functional central limit theorem
and large deviation bounds for the position of the random walker under the annealed law in a high density regime.
As an additional application of our methods, we re-obtain a (slightly improved) ballisticity condition for (the discrete-time version of) the
infection-spread model considered in \cite{KS05}.
Some tools developed in the present paper will be also used in the accompanying article \cite{BHST16b}.

\subsection{Definition of the model and main results}
\label{ss:mainresults}

Denote by $\N = \{1, 2, \ldots\}$ the set of positive integers and let $\Z_+ := \{0 \} \cup \N$.
Fix $d \in \N$ and let $N = (N(x,t))_{x \in \Z^d, t \in \Z_+}$ be a random process 
with each $N(x,t)$ taking values in $\Z_+$,
which we call the \emph{random environment}. 
Let $\alpha: \Z_+ \times \Z^d \to [0,1]$ satisfy
\begin{equation}\label{e:alphaisprob}
\sum_{x \in \Z^d} \alpha(k,x) = 1 \quad \text{ for every $k \in \Z_+$}.
\end{equation}
For a fixed a realization of $N$, the random walker in random environment $X = (X_t)_{t \in \Z_+}$ 
is the Markov chain that, when at position $x \in \Z^d$ at time $t \in \Z_+$,
jumps to $x+z \in \Z^d$ with probability $\alpha(N(x,t),z)$.
Note that the chain is time-inhomogeneous when the random environment is dynamic.
The law of $X$ conditioned on $N$ is called the \emph{quenched law}, and the quenched law averaged over the law of $N$ is called the \emph{annealed law}.

We are interested in the case where $N$ is given by the occupation numbers of a system of simple symmetric random walks in equilibrium.
More precisely, fix $\rho \in (0,\infty)$ and let $(N(x,0))_{x \in \Z^d}$ be an i.i.d.\ collection of Poisson($\rho$) random variables.
From each site $x \in \Z^d$, start $N(x,0)$ independent simple symmetric random walks (which can be lazy or not).
The value of $N(x,t)$, $t > 0$ is then defined as the number of random walks present at $x$ at time $t$.
The process $N(\cdot,t)$ is a Markov chain in equilibrium on the state-space $(\Z_+)^{\Z^d}$.
As already mentioned, $N$ has relatively poor mixing properties; 
for example, it can be shown that $\Cov(N(0,t), N(0,0))$ decays as $t^{-d/2}$ when $t \to \infty$.

Let $|\cdot|$ denote the $\ell^1$-norm on $\Z^d$.
We will make the following assumptions on $\alpha$:
\begin{enumerate}

\item[]\textbf{Assumption (S):}
The set of possible steps
\begin{equation}\label{e:defcS}
\cS := \left\{ x \in \Z^d \colon\, \exists\, k \in \Z_+, \alpha(k,x)>0 \right\}
\end{equation}
is finite. We set $\mathfrak{R} := \max_{x \in \cS} |x|$, which we call the \emph{range} of the random walk.

\item[]\textbf{Assumption (D):}
We assume that
\begin{equation}\label{e:defvbullet}
v_\bullet := \liminf_{k \to \infty} \sum_{x \in \cS} \alpha(k,x) x \cdot e_1 >0,
\end{equation}
where $e_1$ is the first of the canonical base vectors $e_1, \dots, e_d$ of $\mathbb{Z}^d$.

\item[]\textbf{Assumption (R):}
There exists $x_\bullet \in \cS$ satisfying $x_\bullet \cdot e_1>0$ and
\begin{align}\label{e:condxbullet}
\liminf_{k \to \infty} \alpha(k,x_\bullet) >0.
\end{align}

\end{enumerate}

Assumption (D) means that, for sufficiently high particle density, 
the random walker has a local drift in direction $e_1$.
Assumptions (S) and (R) are technical; (S) simplifies the execution of many technical steps while (R) ensures some regularity for $\alpha(k,\cdot)$ over large enough $k \in \N$.
Note that (R) follows from (D) if either $\alpha(k,\cdot)$ is constant for sufficiently large $k$, 
or the random walker moves by nearest-neighbour steps, i.e., $\cS \subset \{x \in \Z^d \colon\, |x| \le 1\}$.

Denote by $\P^\rho$ the joint law of $N$ and $X$ % when the former is started from a product Poisson($\rho$) distribution,
and by $\E^\rho$ the corresponding expectation.
We can now state the main result of the present paper.

\begin{theorem}\label{thm:main}
For every $v_\star \in (0,v_\bullet)$,
there exists a $\rho_{\star} = \rho_{\star} (\alpha, v_\star) <\infty$ large enough such that, for every $\rho \ge \rho_{\star }$,
there exists a $v = v(\alpha, \rho) \in \R^d$ with $v \cdot e_1 \ge v_\star$ and:

\begin{enumerate}
\item[(i)](Law of large numbers)
  \begin{equation}
    \label{e:lln}
   \lim_{t \to \infty} \frac{X_t}{t} = v \quad \text{$\P^\rho-$almost surely.}
  \end{equation}

\item[(ii)](Functional central limit theorem)
  There exists a deterministic covariance matrix $\Sigma = \Sigma(\alpha, \rho)$ 
  such that, under $\P^\rho$,
  \begin{equation}
    \label{e:clt}
    \left(\frac{X_{[nt]} - v [nt]}{\sqrt{n}} \right)_{t \ge 0} \Rightarrow B^\Sigma
  \end{equation}
  where $B^\Sigma$ is a Brownian motion on $\mathbb{R}^d$ with covariance matrix $\Sigma$
  and ``$\Rightarrow$'' denotes convergence in distribution as $n \to \infty$
  with respect to the Skorohod topology.

\item[(iii)](Large deviation bounds) For every $\varepsilon>0$, there exists $c > 0$ such that
  \begin{equation}
    \label{e:ldb}
    \P^\rho \left( \left|\frac{X_t}t - v \right| > \varepsilon \right) \leq {c}^{-1} \exp\{-c (\log t)^{3/2}\} \;\;\; \text{ for all $t \in \N$}.
  \end{equation}

\end{enumerate}
\end{theorem}

Theorem~\ref{thm:main} may be interpreted as follows:
Assumption (D) ensures that the random walker has a positive local drift in direction $e_1$
inside densely occupied regions of $\Z^d$.
Theorem \ref{thm:main} shows that, when the density $\rho$ is large enough,
this behaviour ``takes over'', i.e., the random walker exhibits a macroscopic drift in direction $e_1$,
which introduces enough mixing for a law of large numbers and a central limit theorem to hold.

Note that the matrix $\Sigma$ in item $(ii)$ above might be zero; 
indeed, our assumptions on $\alpha$ do not exclude the case that $X$ is deterministic.
However, $\Sigma$ will be non-zero as soon as $X$ is non-trivial, and it will be non-singular
under mild ellipticity assumptions such as e.g.\ $\sup_{k \in \Z_+} \alpha(k,\pm e_i) >0$ for all $1 \le i \le d$; 
see \eqref{e:formulaSigma}.
The speed of the decay in \eqref{e:ldb} is not optimal, and only reflects the limitations of our methods.

As previously mentioned, one of the biggest obstacles to obtain Theorem~\ref{thm:main} 
are the poor space-time mixing properties of the random environment.
A method to overcome this difficulty in ballistic situations was developed in \cite{HHSST14} 
for the high density regime in one dimension, 
see also \cite{HS15} for a similar approach when the random environment is given by a one-dimensional simple symmetric exclusion process.
However, these results rely on monotonicity properties of the random walker
that are in general not valid in higher-dimensional and/or non-nearest neighbour settings. 
A coupling method (cf.\ \cite{HS14}, \cite{BH14}) can sometimes be used to deal with this problem,
but is limited to cases where $\alpha$ belongs to a set of at most two transition kernels.
Here we follow a different approach, exploiting properties of the random environment 
through more general renormalization and renewal schemes that also bypass the requirement of uniform ellipticity.

As another application of our methods, we provide a short proof of ballisticity for
the one-dimensional discrete-time version of the model for the spread of an infection studied in \cite{KS05}.
In this model, particles can be of two types: healthy or infected.
Fix $\rho \in (0,\infty)$.
At time zero, we place on each site of $\Z$ an independent number of particles, each distributed as a Poisson($\rho$) random variable.
Given the assignment of particles to sites, we declare all particles to the right of the origin to be healthy and all particles 
to its left, including those on the origin, to be infected.
Then the system evolves as follows:
each particle, regardless of its state, 
moves independently as a discrete-time simple symmetric random walk (with a fixed random walk transition kernel), and any healthy particle sharing a site with an infected particle becomes immediately infected.
We are interested in the position $\bar{X}_t$ of the rightmost infected particle at time $t \in \Z_+$.
Still denoting by $\P^\rho$ the underlying probability measure, we obtain:

\begin{proposition}\label{p:infection}
  For any $\rho > 0$, there exist $v > 0$ and $c>0$ such that
  \begin{equation}\label{e:infection}
    \P^\rho \left( \bar{X}_{t} < v t \right) \leq c^{-1} \exp\{- c (\log t)^{3/2} \} \;\; \text{ for all } t \in \N.
  \end{equation}
\end{proposition}
\noindent
The above proposition offers a slight improvement to the deviation bound given in \cite{KS05},
which is an important ingredient in establishing finer results about the infection front.
For example, a similar statement was used in \cite{KS08} to prove a law of large numbers,
and in \cite{BR12} to establish a central limit theorem for (the continuous-time version of) $\bar{X}_t$.

\vspace{10pt}

The rest of the paper is organized as follows.
Section~\ref{ss:proofideas} below contains a short heuristic description of ideas used in our proofs.
In Section~\ref{s:construction}, we give a particular construction of our model with convenient properties.
In Section~\ref{s:renormalization}, we develop a renormalization procedure for general classes of observables,
relying on a key decoupling result for the environment (Theorem~\ref{t:decouple} below) whose proof is given in Appendix~\ref{s:decouple}.
Applications of the renormalization scheme to show ballisticity of the random walker and of the infection front, 
including the proof of Proposition~\ref{p:infection}, 
are discussed in Section~\ref{s:applications}.
Finally, in Section~\ref{s:regmanyRWshighdim} we define and control a regeneration structure for the random walker path
and finish the proof of Theorem~\ref{thm:main}.

Throughout the text, we denote by $c$ a generic positive constant whose value may change at each appearance.
These constants may depend on all model parameters discussed above but,
in Section~\ref{s:renormalization} and in Appendix~\ref{s:decouple}, 
they will not be allowed to depend on $\rho$, as we recall in the beginning of these sections.

\vspace{10pt}
\noindent
{\bf Acknowledgments.}
OB acknowledges the support of the French Ministry of Education through the ANR 2010 BLAN 0108 01 grant.
MH was partially supported by CNPq grants 248718/2013-4 and 406659/2016-8 and by ERC AG ``COMPASP''.
RSdS was supported by the German DFG project KO 2205/13-1.
AT was supported by CNPq grants 306348/2012-8 and 478577/2012-5 and by FAPERJ grant 202.231/2015.
OB, MH and RSdS thank IMPA for hospitality and financial support.
AT and MH thank the CIB for hospitality and financial support.
RSdS thanks the ICJ for hospitality and financial support.
The research leading to the present results benefited from
the financial support of the seventh Framework Program of the European Union (7ePC/2007-2013), 
grant agreement n\textsuperscript{o}266638.
Part of this work was carried on while MH was on a sabbatical year on the University of Geneva. 
He thanks the mathematics department of this university for the financial support.
OB and AT thank the University of Geneva for hospitality and financial support.

\subsection{Proof ideas}
\label{ss:proofideas}

The proof of Theorem~\ref{thm:main} is split into two main steps that can be informally described as: 
a ballisticity condition and a renewal decomposition.
They are performed respectively in Sections~\ref{s:applications} and \ref{s:regmanyRWshighdim}.
Let us now describe them in more detail.

Our first result for the random walker described above is reminiscent of the 
$(T')$-condition of Sznitman (see \cite{Sznitman04}):
in Theorem~\ref{thm:ballisticity}, we show that, for $\rho$ large enough,
$X_t \cdot e_1$ diverges to infinity in a strong sense, i.e., the random walker is ballistic.
This is done via a renormalization argument.
Once ballisticity has been established, our intuition tells us that, as time passes, 
the random walker will see ``fresh environments'' since the particles of the random environment have no drift;
this informal description is made precise by defining a regeneration structure for the path of the random walker.
Here this step must be performed differently from \cite{HHSST14} because of the higher dimensions and non-nearest neighbour transition kernels. 
Moreover, because of the lack of monotonicity, the tail of the regeneration time must also be controlled differently.

Proposition~\ref{p:infection} is proved using a similar argument as for Theorem~\ref{thm:ballisticity}.
First, the problem is reduced to showing that, with large probability, we can frequently find particles 
near $\bar{X}_t$, cf.\ Lemma~\ref{l:infection_reduction}. Indeed, this implies the existence of a density of times where $\bar{X}_t$ behaves as a random walk with a drift, which is enough because it always dominates a simple symmetric random walk.
The reduced problem can then be tackled using the renormalization procedure, driving not $\rho$ to infinity but the size of the window around $\bar{X}_t$ where we look for particles. See Section~\ref{ss:infection}.

\section{Construction}
\label{s:construction}

In this section, we introduce a construction of the environment of simple random walks in terms of a Poisson point process of trajectories as in \cite{HHSST14}.
This construction provides a convenient way to explore certain independence properties of the environment. 
We also provide a construction for the random walker and discuss positive correlations of certain monotone observables of the environment (cf.~Proposition~\ref{prop:FKG} below).

Define the set of trajectories
\begin{equation}
  \label{e:def_W}
  W = \Big\{ w:\Z \to \Z^d \colon\, |w(i + 1) - w(i)| \le 1 \;\; \forall \; i \in \Z \Big\}.
\end{equation}
Note that the trajectories in $W$ are allowed to jump in any canonical direction, as well as to stay put.
We endow the set $W$ with the $\sigma$-algebra $\mathcal{W}$ generated by the canonical coordinates $w \mapsto w(i)$, $i \in \Z$.

Let $(S^{z,i})_{z \in \Z^d, i \in \N}$ be a collection of independent random elements of $W$,
with each $S^{z,i} = (S^{z,i}_\ell)_{\ell \in \Z}$ distributed as
a double-sided simple symmetric random walk on $\Z^d$ started at $z$,
i.e., the past $(S^{z,i}_{-\ell})_{\ell \ge 0}$ and future $(S^{z,i}_{\ell})_{\ell \ge 0}$
are independent and distributed as a simple symmetric random walk on $\Z^d$ started at $z$ (lazy or not).

For a subset $K \subset \mathbb{Z}^{d}\times \Z$, denote by $W_K$ the set of trajectories in $W$ that intersect $K$,
i.e., $W_K := \{w \in W \colon\, \exists\, i \in \Z, (w(i),i) \in K\}$.
This allows us to define the space of point measures
\begin{equation}
  \label{e:Omega}
  \Omega = \Big\{ \omega = \sum_i \delta_{w_i};\; w_i \in W \text{ and } \omega(W_{\{y\}}) < \infty \text{ for every } y \in \Z^d \times \Z \Big\},
\end{equation}
endowed with the $\sigma$-algebra generated by the evaluation maps $\omega \mapsto \omega(W_K)$, $K \subset \Z^d \times \Z$.

Fix $\rho \in (0,\infty)$ and let
$N(x,0)$, $x \in \Z^d$ be i.i.d.\ Poisson($\rho$) random variables.
Defining the random element $\omega \in \Omega$ by
\begin{equation}\label{e:defomega}
\omega := \sum_{z \in \Z^d} \sum_{i \le N(z,0)} \delta_{S^{z,i}},
\end{equation}
it is straightforward to check that $\omega$ is a Poisson point process on $\Omega$
with intensity measure $\rho \mu$, where
\begin{equation}
\label{e:defmu}
  \mu = \sum_{z \in \Z^d} P_z
\end{equation}
and $P_z$ is the law of $S^{z,1}$ as an element of $W$.
Setting then
\begin{equation}\label{e:defN}
N(y) := \omega(W_{\{y\}}), \text{ for $y \in \mathbb{Z}^d \times \mathbb{Z}$},
\end{equation}
we may verify that $N$ has the distribution described in Section \ref{s:intro}.

We enlarge our probability space to support i.i.d.\ random variables $U_y$, $y \in \Z^d \times \Z$ sampled independently from $\omega$, where each $U_y$ is uniformly distributed in the interval $[0,1]$.
We then define $\P^\rho$ to be the joint law of $\omega$ and $U = (U_y)_{y \in \Z^d \times \Z}$.
Our configuration space may be thus identified as $\overline{\Omega} := \Omega \times [0,1]^{\Z^d \times \Z}$,
equipped with the product $\sigma$-algebra.

To define our random walker, recall Assumption (R) and let
\begin{equation}\label{e:defpbulletk}
p_\bullet(k) := \inf_{\ell \ge k} \alpha(\ell,x_\bullet), \quad k \in \Z_+.
\end{equation}
For each $k \in \Z_+$, fix a partition of $[0,1]$ into intervals $I^k_x$, $x \in \cS$ such that $|I^k_x| = \alpha(k, x)$
and $[0,p_\bullet(k)] \subset \bigcap_{\ell \ge k} I^\ell_{x_\bullet}$.
Finally, for $y \in \Z^d \times \Z$, we define $Y^y = (Y^y_\ell)_{\ell \in \Z_+}$ by
\begin{equation}\label{e:defY}
Y^y_0 =y, \qquad Y^y_{\ell+1} = Y^y_\ell + \sum_{x \in \cS} (x, 1) \mathbbm{1}_{\left\{ U_{Y^y_\ell} \in I^{N(Y^y_\ell)}_x \right\}}, \;\; \ell \ge 0.
\end{equation}
For $y = (x,t) \in \mathbb{Z}^d \times \mathbb{Z}$ we write $X^y_\ell$ for the projection of $Y^y_\ell$ into $\mathbb{Z}^d$, i.e., $Y^y_\ell = Y_\ell^{(x,t)} = (X^y_\ell, \ell+t)$.
When $y=0$ we omit it from the notation.
One may verify that the random walker $X = (X_\ell)_{\ell \in \Z_+}$ is indeed distributed as described in Section \ref{s:intro}.

We discuss next an important property of our random environment:
the FKG inequality (cf.\ e.g.\ \cite{Liggett85} Corollary 2.12 p.\ 78).
It states that monotone functions of $\omega$ are positively correlated.
This result will be used in the proof of Lemma~\ref{l:no_top},
which is an important ingredient to control the tail of the regeneration time constructed in Section~\ref{s:regmanyRWshighdim}.
We first need the following definition.

\begin{definition}\label{d:monotone_FKG}
A measurable function $f:\Omega \to \R$
is called \emph{non-decreasing} if
$f(\omega') \ge f(\omega)$ whenever $\omega', \omega \in \Omega$
satisfy $\omega'(B) \ge \omega(B)$ for all $B \in \cW$, and it is called 
\emph{non-increasing} if $-f$ is non-decreasing.
If either $f$ or $-f$ is non-decreasing, we say that $f$ is \emph{monotone}.
An event $A \in \sigma(\omega)$ is said to be non-decreasing, non-increasing or monotone if the corresponding property is satisfied by its indicator function $\mathbbm{1}_A$.
\end{definition}

The inequality reads as follows.
\begin{proposition}[FKG inequality]
\label{prop:FKG}
Let $f, g:\Omega \to \R$ be bounded measurable functions that are either both non-decreasing or both non-increasing. 
Then
\begin{equation}\label{e:FKG}
\E^\rho \left[f(\omega) g(\omega) \right] \ge \E^\rho \left[ f(\omega) \right] \E^\rho \left[ g(\omega) \right].
\end{equation}
\end{proposition}
\begin{proof}
One may follow the proof of Theorem~3.1 in \cite{T09}.
\end{proof}

We extend the notion of monotonicity to functions defined on $\overline{\Omega}$ as follows.
\begin{definition}\label{d:monotone}
A measurable function $f:\overline{\Omega} \to \R$
is called non-decreasing, non-increasing or monotone if, for all $U \in [0,1]^{\Z^d}$, 
the function $\omega \mapsto f(\omega, U)$ satisfies the same property in the sense of Definition~\ref{d:monotone_FKG},
and analogously for events in $\sigma(\omega, U)$.
\end{definition}

\begin{remark}
Note that monotone functions in the sense of Definition~\ref{d:monotone} are \emph{not} necessarily positively correlated under $\P^\rho$: consider e.g.\ the indicator functions of the events $\{U_0 \in (0,1/2)\}$ and $\{U_0 \in (1/2, 1)\}$.
However, such monotone functions are positively correlated under the conditional law given $U$, 
as can be deduced from Proposition~\ref{prop:FKG}.
\end{remark}

We give next a few other useful definitions.
For a measurable function $g: \overline{\Omega} \to E$ (with $E$ some measurable space),
we will abuse notation by writing $g$ to refer also to the random variable $g(\omega,U)$,
distributed according to the push-forward of $\P^\rho$. 
\begin{definition}
\label{d:supported}
We say that the function $g:\overline{\Omega} \to E$ is \emph{supported on the set} $K \subset \Z^d \times \Z$ if
\begin{equation}\label{e:defsupportedK}
g(\omega, U) \in \sigma(N(y), U_y \colon\, y \in K).
\end{equation}
\end{definition}

For $y=(x,n) \in \Z^d \times \Z$ and $w \in W$,
define the space-time translation $\theta_y w$ as
\begin{equation}\label{e:defthetaw}
\theta_y w(i) := w(i-n) + x, \;\;\; i \in \Z.
\end{equation}
For $A \subset W$, $\theta_y A$ is defined analogously, i.e., $\theta_y A := \bigcup_{w \in A} \{\theta_y w\}$.
We may then define space-time translations operating on $\overline{\Omega}$ as follows.
For $(\omega,U) \in \overline{\Omega}$, let
\begin{equation}\label{e:def_theta}
\begin{aligned}
\theta_y (\omega,U) & := (\theta_y \omega, \theta_y U) \\%\;\; \text{ where }\\
\text{ where } \;\; (\theta_y \omega)(A) & := \omega(\theta_y A) \;\; \forall \; A \in \cW, \;\;\;\; (\theta_y U)_u := U_{y+u} \;\; \forall \; u \in \Z^d \times \Z.
\end{aligned}
\end{equation}
The translations of a measurable function $g:\overline{\Omega} \to E$ 
are then defined by setting
\begin{equation}
\label{e:def_gy}
g_y = \theta_y g := g \circ \theta_y.
\end{equation}
Note that $\theta_y N(u) = N(y+u)$, and that the law of $(\omega,U)$ is invariant with respect to the space-time translations, i.e., $\theta_y (\omega,U)$ is distributed as $(\omega,U)$ under $\P^\rho$ for any $y \in \Z^d \times \Z$.
In particular, the law of $Y^y-y$ in \eqref{e:defY} does not depend on $y$ since $Y^y = y + \theta_y Y$.

%%%% End of Section Intromany %%%%%%%

%%%%\subfile{renormalizationmany}%%%%%%
%%%% Subfile replaced by the text below %%%%
\section{Renormalization}
\label{s:renormalization}
In this section, we develop an important tool in the analysis of our model,
namely, a multi-scale renormalization scheme.
We will keep the setup reasonably general so that it may be used in future applications.
An important consequence of the technique developed here
is the ballisticity of the random walker (cf.\ Theorem~\ref{thm:ballisticity}),
which is an essential ingredient for proving Theorem~\ref{thm:main}.
All constants in this section  will be \emph{independent} of $\rho$, but may depend on other parameters of the model.

\subsection{General procedure}
To describe the renormalization procedure, we introduce the sequence of scales
\begin{equation}
  \label{e:Lk}
  L_0 = 10^{50} \qquad \text{and} \qquad L_{k + 1} = \lfloor L_k^{1/2} \rfloor L_k, \text{ for $k \geq 0$}.
\end{equation}
The choice of constants $10^{50}$ and $1/2$ appearing above is not crucial;
many other choices would have been equally good for our purposes.
Note that
\begin{equation}
\label{e:lowerbound_Lk}
  L_k^{1/2} \geq \lfloor L_k^{1/2} \rfloor \geq \tfrac{1}{2} L_k^{1/2} \;\;\;\; \text{ for all $k\ge 0$}.
\end{equation}
Fix $\mathfrak{R} \in \N$. In the relevant applications, $\mathfrak{R}$ will be taken as in Assumption (S).
\todo{do we ever use $\mathfrak{R}$ for something else?}
Given a scale $k$, we will consider translations of the space-time boxes
\begin{equation}
  \label{e:Bk}
  B_{k,0} = \left\{ [-2 \mathfrak{R} L_k, 3 \mathfrak{R} L_k)^d \times [0, L_k) \right\} \cap (\Z^d \times \Z).
\end{equation}
More precisely, let
\begin{equation}
  M_k = \{k\} \times (\Z^d \times \Z), \;\;\; k \geq 0,
\end{equation}
be the set of indices of scale $k$ and, for $\hat{k} \ge 0$, let
\begin{equation}
M_{\geq \hat{k}} = \bigcup_{k \geq \hat{k}} M_k
\end{equation}
be the set of indices of all scales greater or equal to $\hat{k}$.
For $m = (k,y) \in M_k$, we define the corresponding translation of the box $B_{k,0}$
\begin{equation}
  \label{e:Bm}
  B_m = B_{k,0} + L_k y.
\end{equation}
The base of the box $B_{k,0}$ is given by the set $I_{k,0} = ([0,\mathfrak{R} L_k)^d \times \{0\}) \cap (\Z^d \times \Z)$
and its corresponding translations are
\begin{equation}
  \label{e:Im}
  I_m = I_{k,0} + L_k y, \text{ for $m = (k, y) \in M_k$},
\end{equation}
see Figure~\ref{f:Bm}.
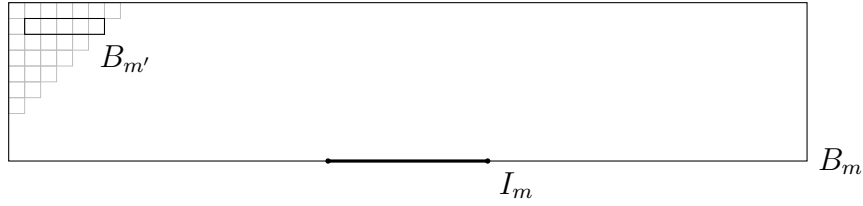
\begin{figure}[h]
  \centering
  \begin{tikzpicture}[scale=.7]
    \draw[very thick] (6,0) -- (9, 0); \node at (9, 0) [below right] {$I_m$};
    \fill (6,0) circle (0.05); \fill (9,0) circle (0.05);
    \foreach \x in {0,...,6}
    {   \foreach \y in {0,...,\x} {
        \draw[color=gray!50, very thin] (1.8 - .3 * \x, 3 - .3 * \y) rectangle (2.1 - .3 * \x, 2.7 - .3 * \y);
      }
    }
    \node at (1.5, 2.4) [below right] {$B_{m'}$};
    \draw (0.3, 2.4) rectangle (1.8, 2.7);
    \draw (0,0) rectangle (15, 3); \node at (15, 0) [right] {$B_m$};
  \end{tikzpicture}
  \caption{The box $B_m$ for some $m \in M_{k+1}$. At the bottom we picture the corresponding base $I_m$ and, in the upper left corner, a box $B_{m'}$ in the previous scale $k$.}
  \label{f:Bm}
\end{figure}

Having this in place, we introduce the following definition.
\begin{definition}
  \label{d:adapted}
  Fix $\hat{k} \geq 0$ and a collection of events $(A_m)_{m \in M_{\geq \hat{k}}}$.
  We say that this collection is \emph{adapted} if the indicator function of $A_m$
  is supported in $B_m$ (as in Definition~\ref{d:supported})
  for each $m \in M_{\geq \hat{k}}$.
 \end{definition}

We aim to bound the probability of certain events $A_m$ inductively in $k$.
For this, we will need another definition, concerning the occurrence of $A_m$ in consecutive scales.
\begin{definition}
Fix $\hat{k} \geq 0$ and a collection of events $(A_m)_{m \in M_{\geq \hat{k}}}$.
This collection is said to be \emph{cascading} if, for every $k \geq \hat{k}$ and $m \in M_{k + 1}$, we have
\begin{equation}\label{e:cascading}
A_m \subseteq \bigcup_{m_1 \underset{m}\leftrightarrow m_2} A_{m_1} \cap A_{m_2},
\end{equation}
where $m_1 \leftrightarrow_m m_2$ stands for pairs of  indices $m_1, m_2 \in M_k$ such that $B_{m_1}, B_{m_2} \subseteq B_{m}$ and such that the vertical distance between the boxes $B_{m_1}$, $B_{m_2}$ is at least $L_k$.
\end{definition}
\noindent
In the definition above, if $m_1 = (k,y_1)$ and $m_2=(k,y_2)$ with $y_1 = (x_1,t), \; y_2 = (x_2, s) \in \mathbb{Z}^{d} \times \mathbb{Z} $ we say that the vertical distance between the boxes $B_{m_1}$ and $B_{m_2}$ is equal to $[(|t- s|-1) L_k + 1 ]^+$.

Intuitively speaking, the above definition says that the occurrence of $A_m$ implies that two similar events happened in well-separated boxes of the smaller scale.
The imposition that the boxes indexed by $m_1$ and $m_2$ in \eqref{e:cascading} are vertically separated
will be useful to decouple the events $A_{m_1}$ and $A_{m_2}$ via Theorem~\ref{t:decouple}.
Examples of cascading events will be given in Section~\ref{ss:cascading}.

Given a family $(A_m)_{m \in M_{\geq \hat{k}}}$, we will be interested in the following quantities:
\begin{equation}\label{e:pk}
p_k(\rho) = \sup_{m \in M_k} \mathbb{P}^{\rho} (A_m), \quad k \ge \hat{k}.
\end{equation}
Let us also denote
\begin{equation}\label{e:defiota}
\iota_{\hat{k}} := \exp \bigg\{ 2 \sum_{k \ge \hat{k}} L_k^{-1/16} \bigg\} \in (0,\infty).
\end{equation}
The next theorem is the main result that we will use in order to bound $p_k(\rho)$.
\begin{theorem}\label{t:pk_decay}
For any $\gamma \in (1,3/2]$, there exists $k_o = k_o(d,\gamma) \in \Z_+$ such that the following holds.
Fix $\hat{k} \geq k_o$ and a collection $(A_m)_{m \in M_{\geq \hat{k}}}$ that is adapted and cascading.
Assume that the $A_m$'s are either all non-increasing or all non-decreasing and that, for some $\hat{\rho} \ge L^{-1/16}_{\hat{k}}$,
\begin{equation}\label{e:kprime_rhoprime}
p_{\hat{k}}(\hat{\rho}) \leq \exp\{-(\log L_{\hat{k}})^{\gamma}\}.
\end{equation}
Then, writing $\rho_* := \iota_{\hat{k}} \hat{\rho}$ and $\rho_{**} := \iota_{\hat{k}}^{-1} \hat{\rho}$, for all $k \ge \hat{k}$ we have
\begin{equation}\label{e:pk_decay}
p_k (\rho) \leq \exp\{-(\log L_{k})^{\gamma}\}
\;\;\;\left\{
\begin{array}{ll}
 \forall\, \rho\ge \rho_* & \text{ in the non-increasing case,}\\
 \forall\, \rho \le \rho_{**} & \text{ in the non-decreasing case.}
\end{array}
\right.
\end{equation}
\end{theorem}
\noindent
The upper bound $3/2$ appearing in Theorem~\ref{t:pk_decay} is not sharp;
any number $\beta$ satisfying $(3/2)^\beta < 2$ would suffice (see \eqref{e:defk_o2} below).

The statement of the previous theorem has two different cases,
depending on whether the events $A_m$ are non-increasing or non-decreasing.
All applications considered in this paper concern non-increasing events,
but we choose to keep the exposition general in order to be able to use our results in the accompanying paper \cite{BHST16b}.

One of the main ingredients for the proof of Theorem \ref{t:pk_decay}
is a recursion inequality for $p_k$, cf.\ Lemma~\ref{l:induction} below.
As the cascading property suggests,
the key to obtain such a recursion is to decouple pairs of events $A_{m_1}$ and $A_{m_2}$ supported in boxes that are well-separated in time.
Recall however that the environment of simple random walks, being conservative, presents poor mixing properties,
which makes decoupling hard.
We overcome this difficulty using a ``sprinkling technique'',
which consists in performing a change in the density of particles in the environment in order to blur the dependency between such events.
Thus, up to an error term, we bound $\P^{\rhoprime}(A_{m_1} \cap A_{m_2})$ by the product $\P^{\rho}(A_{m_1}) \P^{\rho}(A_{m_2})$, where $\rho$ is slightly different from $\rhoprime$.
This is the content of Theorem~\ref{t:decouple} below,
which has different statements for the cases where the events $A_m$ are non-increasing or non-decreasing.
\begin{theorem}
\label{t:decouple}
There exist constants $n_o \in \N$, $C_o \ge 1$ and $c_o > 0$, depending only on $d$ and on the law of $S^{0,1}$,
such that the following holds.
Let $B = ([a, b]^d \times [n, n']) \cap \Z^d \times \Z$ be a space-time box satisfying $n \geq n_o$,
and let $D = \Z^d \times \Z_-$ be the space-time
lower half-space.
Let $f_1, f_2 \colon\, \overline{\Omega} \to [0,1]$
be measurable functions supported respectively in $D$ and $B$ (cf.\ Definition~\ref{d:supported}).
Denote by $\textnormal{diam}(B)$ the diameter of $B$.
Then, for all $\rho > 0$:
\begin{enumerate}
\item[(a)] If both $f_1$ and $f_2$ are non-increasing (cf.\ Definition~\ref{d:monotone}), then
\begin{equation}
\label{e:dec_NI}
\mathbb{E}^{\rho(1+n^{-1/16})}[f_1 f_2] \leq \mathbb{E}^{\rho(1+n^{-1/16})}[f_1] \,\,
\mathbb{E}^{\rho}[f_2] + C_o \big(\textnormal{diam}(B) + n\big)^d \,e^{-2 c_o \rho n^{1/8}}.
\end{equation}

\item[(b)] If both $f_1$ and $f_2$ are non-decreasing (cf.\ Definition~\ref{d:monotone}), then
\begin{equation}
\label{e:dec_ND}
\mathbb{E}^{\rho}[f_1 f_2] \leq \mathbb{E}^{\rho}[f_1] \,\,
\mathbb{E}^{\rho(1+n^{-1/16})}[f_2] + C_o \big(\textnormal{diam}(B) + n\big)^d \,e^{-2 c_o \rho n^{1/8}}.
\end{equation}
\end{enumerate}
\end{theorem}
\noindent
The proof of Theorem~\ref{t:decouple} is given in the Appendix~\ref{s:decouple},
and is very similar to the proof of Theorem~C.1 in \cite{HHSST14}.

We may now identify the constant $k_o$ appearing in Theorem~\ref{t:pk_decay}.
Fix $\gamma \in (1,3/2]$ and let $n_o, C_o, c_o$ as given by Theorem~\ref{t:decouple}.
Then fix $k_o = k_o(d,\gamma) \in \Z_+$ such that
\begin{equation}\label{e:defk_o1}
L_{k_o} \ge n_o
\end{equation}
and
\begin{equation}\label{e:defk_o2}
C_o L_k^{2d+1} \left( e^{-(2-(3/2)^{3/2}) (\log L_k)^{\gamma}} + e^{-c_o L_k^{1/16} + (3/2)^{3/2} (\log L_k)^{3/2}}\right) < 1 \;\;\;\forall\; k \ge k_o.
\end{equation}

As anticipated, Theorem~\ref{t:decouple} leads to the following recursion inequality for $p_k$.
\begin{lemma} \label{l:induction}
Fix $\gamma \in (1,3/2]$ and let $k_o \in \Z_+$ as in \eqref{e:defk_o1}--\eqref{e:defk_o2}.
Fix $\hat{k} \geq k_o$ and a collection $(A_m)_{m \in M_{\geq \hat{k}}}$ that is adapted and cascading.
Assume that the $A_m$'s are either all non-increasing or all non-decreasing.
For fixed $\rho >0$ and $k \ge \hat{k}$, define
\begin{equation}\label{e:barrho}
\rhoprime = \left\{
\begin{array}{ll}
\rho (1 + L^{-1/16}_k) & \text{ in the non-increasing case,}\\
\rho (1 - L^{-1/16}_k) & \text{ in the non-decreasing case.}
\end{array}
\right.
\end{equation}
Then we have
\begin{equation}\label{e:induction}
p_{k+1}(\rhoprime) \leq C_o L_k^{2d+1} \big( p_k(\rho)^2 + \exp\{-c_o \rho L_k^{1/8}\} \big),
\end{equation}
where $C_o, c_o$ are as in Theorem~\ref{t:decouple}.
\end{lemma}
\begin{proof}
We start with the case when the $A_m$'s are all non-increasing.
Using that the $A_m$'s are adapted and cascading and that, by \eqref{e:defk_o1}, $L_k \ge n_o$,
we apply Theorem~\ref{t:decouple} to the indicator functions of $A_{m_1}$, $A_{m_2}$
to obtain
\begin{equation}
\begin{split}
p_{k+1} & (\rhoprime) \;\; = \;\; \sup_{m \in M_{k+1}} \P^{\rhoprime} (A_m) \leq \sup_{m \in M_{k+1}} \sum_{m_1 \underset{m}\leftrightarrow m_2} \P^{\rhoprime} (A_{m_1} \cap A_{m_2})\\
      & \begin{array}{e}
        & \overset{\eqref{e:Lk}}\leq \;\; & \lfloor L_k^{1/2} \rfloor^{2(d+1)} \sup_{m \in M_{k+1}} \sup_{\smash{m_1 \underset{m}\leftrightarrow m_2}} \P^{\rhoprime} (A_{m_1} \cap A_{m_2})\\
        & \overset{\text{Theorem~\ref{t:decouple}}}\leq & L_k^{d+1} \sup_{m_1, m_2 \in M_{k}}  \big( \P^{\rhoprime} (A_{m_1}) \P^{\rho} (A_{m_2}) + C_o L_k^d \exp\{-2 c_o \rho L_k^{1/8}\} \big)\\[2mm]
        & \overset{\text{$A_m$'s non-increas.}}\leq & C_o L_k^{2 d+1} \big(p_k(\rho)^2 + \exp\{-2 c_o \rho L_k^{1/8}\}\big),
      \end{array}
    \end{split}
  \end{equation}
This finishes the proof of \eqref{e:induction} in the first case.

Now assume that the events $A_m$ are all non-decreasing.
As before, we can estimate
\begin{align}
p_{k+1}(\rhoprime)
& \quad \;\;= \sup_{m \in M_{k+1}} \P^{\rhoprime} (A_m) \leq \sup_{m \in M_{k+1}} \sum_{m_1 \underset{m}\leftrightarrow m_2} \P^{\rhoprime} (A_{m_1} \cap A_{m_2})\nonumber\\
& \leq C_o L_k^{2 d+1} \Big( p_{k}(\rho)^2 + \exp\{ -2 c_o \rhoprime L_k^{1/8}\} \Big).
\end{align}
Since, by the definition of $L_0$, $\bar{\rho} \ge \rho/2$, \eqref{e:induction} follows.
\end{proof}

Now that we know how large the sprinkling should be as we move from scale $k+1$ to $k$
in order to obtain a good recursive inequality for the $p_k$'s, we will introduce a sequence of densities $\rho_k$.

Given $\rho_{\hat{k}} > 0$, define $\rho_k$ for $k \geq \hat{k}$ recursively by setting
\begin{equation}\label{e:rhok}
\rho_{k+1} = \left\{
\begin{array}{ll}
\rho_k (1 + L_{k}^{-1/16}) & \text{ when the $A_m$'s are all non-increasing,}\\
\rho_k (1 - L_{k}^{-1/16}) & \text{ when the $A_m$'s are all non-decreasing.}
\end{array}
\right.
\end{equation}

Note that, with the above definition, when the $A_m$'s are non-increasing,
\begin{equation}\label{e:rho_star}
\log \rho_k = \log \rho_{\hat{k}} + \sum_{i=\hat{k}}^k \log(1 + L_i^{-1/16}) \leq \log \rho_{\hat{k}} +  \sum_{i=\hat{k}}^\infty L_i^{-1/16} < \log (\iota_{\hat{k}} \rho_{\hat{k}}) < \infty
\end{equation}
while, when the $A_m$'s are non-decreasing, since $e^{-2x} \leq 1 - x$ for all $x \in (0, L_0^{-1/16})$,
\begin{equation}\label{e:rho_star_star}
\log \rho_k \geq \log \rho_{\hat{k}} - 2 \sum_{i=\hat{k}}^\infty L_i^{-1/16} = \log ( \iota_{\hat{k}}^{-1} \rho_{\hat{k}}) > -\infty.
\end{equation}
This shows that the sequence of densities $\rho_k$ is not asymptotically trivial.

We are now in position to prove Theorem~\ref{t:pk_decay}.
\begin{proof}[Proof of Theorem~\ref{t:pk_decay}]
Given $\gamma \in (1,3/2]$, take $k_o$ as in \eqref{e:defk_o1}--\eqref{e:defk_o2}.
The first step in the proof is to show how one can use \eqref{e:induction}
to transport the bound $p_k(\rho_k) \leq \exp\{-(\log L_{k})^{\gamma}\}$ to the scale $k+1$.
This can be summarized by saying that:
\begin{display}
\label{e:assume_induction}
if \eqref{e:induction} holds with $\rho = \rho_k \geq L_k^{-1/16}$, then\\
$p_k(\rho_k) \leq \exp\{-(\log L_k)^{\gamma}\}$ implies $p_{k+1}(\rho_{k+1}) \leq \exp\{-(\log L_{k+1})^{\gamma}\}$.
\end{display}
To see why this is true, let us first use \eqref{e:induction} in order to estimate
\begin{align}\label{e:conseq_induc0}
\frac{p_{k+1}(\rho_{k+1})}{\exp\{-(\log L_{k+1})^{\gamma}\}}
\leq C_o L_k^{2d+1} \left( e^{-2 (\log L_k)^{\gamma}} + e^{-c_o \rho_k L_k^{1/8} } \right) \exp\{(\log L_{k+1})^{\gamma}\}.
\end{align}
Now, since $\rho_k \ge L_k^{-1/16}$ and by \eqref{e:Lk},
\eqref{e:conseq_induc0} is at most
\begin{equation}\label{e:conseq_induc}
C_o L_k^{2d+1} \left( e^{-(2 - (3/2)^{3/2}) (\log L_k)^{\gamma} } + e^{-c_o L_k^{1/16} + (3/2)^{3/2} (\log L_k )^{3/2} } \right)
\end{equation}
which is smaller than $1$ by \eqref{e:defk_o2}, proving \eqref{e:assume_induction}.

Let us now see how \eqref{e:pk_decay} follows from \eqref{e:assume_induction} and Lemma~\ref{l:induction}.
Let $\rho_{\hat{k}} = \hat{\rho}$ (from \eqref{e:kprime_rhoprime})
and define $\rho_k$ for $k \geq \hat{k}$ through \eqref{e:rhok}.
We claim that, for all $k \ge \hat{k}$,
\begin{equation}\label{e:ineqwithk}
p_k(\rho_k) \le \exp \left\{ - (\log L_k)^\gamma \right\}.
\end{equation}
Then \eqref{e:pk_decay} follows since, by the definition of $\iota_{\hat{k}}$ and \eqref{e:rho_star}--\eqref{e:rho_star_star},
if $\rho \ge \iota_{\hat{k}} \rho_{\hat{k}}$ and the $A_m$'s are non-increasing
(resp.\ $\rho \le \iota_{\hat{k}}^{-1} \rho_{\hat{k}}$ and the $A_m$'s are non-decreasing),
then $p_k(\rho) \le p_k(\rho_k)$.
Thus we only need to prove \eqref{e:ineqwithk}.

To this end, we first claim that, for all $k \ge \hat{k}$,
\begin{equation}\label{e:rhokdoesnotdecaytoofast}
\rho_k \ge L_k^{-1/16}.
\end{equation}
Indeed, if the $A_m$'s are all non-increasing, then $\rho_k \ge \rho_{\hat{k}} \ge L_{\hat{k}}^{-1/16} \ge L_k^{-1/16}$ by \eqref{e:rhok}
and our definition of $\rho_{\hat{k}}$ while, if the $A_m$'s are all non-decreasing,
then \eqref{e:rhokdoesnotdecaytoofast} follows by induction using \eqref{e:rhok},
\eqref{e:Lk} and the assumption that $\hat{\rho} \geq L_{\hat{k}}^{-1/16}$.

Let us now prove \eqref{e:ineqwithk} by induction on $k$.
The case $k=\hat{k}$ holds by hypothesis.
Assume now that \eqref{e:ineqwithk} holds for some $k \ge \hat{k}$.
Noting that $\rho_{k+1} \ge L^{-1/16}_{k+1}$ by \eqref{e:rhokdoesnotdecaytoofast} and that
\eqref{e:induction} holds for $\rho_k$ and $\rho_{k+1}$ replacing $\rho$ and \ $\bar{\rho}$ respectively (because the relation between $\rho_{k+1}$ and $\rho_k$ is exactly as for $\bar{\rho}$ and $\rho$ in Lemma~\ref{l:induction})
we conclude by \eqref{e:assume_induction}, that \eqref{e:ineqwithk} also holds with $k+1$ replacing $k$.
This concludes the induction step and the proof of the theorem.
\end{proof}

\subsection{Constructing cascading events}
\label{ss:cascading}
We provide in this section a systematic way to construct certain collections of cascading events
based on averages of functions of the random environment along Lipschitz paths.
Our ultimate goal is to obtain Corollary~\ref{cor:renormnoscales} below,
which provides in this context a short-cut to ballisticity-type results with minimal 
reference to the bulkier technical setup of the previous section.

Let us first describe the type of paths that we will consider.
We say that a function $\sigma:\Z \to \Z^d$ is $\mathfrak{R}$-Lipschitz
if for any $x, y \in \mathbb{Z}$, we have $|\sigma(x) - \sigma(y)| \leq \mathfrak{R} |x - y|$.
The set of $\mathfrak{R}$-Lipschitz paths is defined as
\begin{equation}\label{e:deffrakS}
\mathfrak{S} := \left\{ \sigma:\Z_+ \to \Z^d \colon\, \sigma(0) =0, \left|\sigma(i+1) - \sigma(i)\right| \le \mathfrak{R}  \;\forall \; i \in \Z_+ \right\},
\end{equation}
We will further restrict the class of paths using a function $H : \overline{\Omega} \times \mathbb{Z}^d \to \{0,1\}$
as follows.
Given such a function $H$ and a path $\sigma : [n, \infty) \to \mathbb{Z}^d$, we define, for $t \geq n$,
\begin{equation}
  h_\sigma(t) := H( \theta_{(\sigma(t), t)} (\omega, U), \sigma(t + 1) - \sigma(t)).
\end{equation}
The interpretation is that $h_\sigma(t) = 1$ if and only if the jump $\sigma(t+1)-\sigma(t)$ is allowed by the random environment according to the rule $H$. The formal definition is as follows.

\begin{definition}
Given a box index $m \in M_k$, we say that an $\mathfrak{R}$-Lipschitz function $\sigma: [n, \infty) \cap \mathbb{Z} \to \mathbb{Z}^d$ is an $m$-crossing if $(\sigma(n), n) \in I_m$ (recall the definition of $B_m$ and $I_m$ in \eqref{e:Bm} and \eqref{e:Im}).
In addition, if for $H : \overline{\Omega} \times \mathbb{Z}^d \to \{0,1\}$ we have $h_{\sigma}(j) = 1$ for all $j \in [n, n+L_k) \cap \mathbb{Z}$, we say that $\sigma$ is an $(m,H)$-crossing.
\end{definition}
\noindent
Figure~\ref{f:Bm_cross} illustrates an $(m, H)$-crossing.
Note that an $(m, H)$-crossing does not exit $B_m$ through its sides.
Furthermore, if $m \in M_{k+1}$, every $(m,H)$-crossing induces $L_{k+1}/L_k$ $(m_i,H)$-crossings for $m_i \in M_k$ such that $B_{m_i} \subset B_m$. Note that in the particular case where $H \equiv 1$ the notions of $m$-crossing and $(m, H)$-crossing coincide.

\begin{remark}
In the remainder of this paper, we will only be interested in applications where $H \equiv 1$, which implies that every $\mathfrak{R}$-Lipschitz function starting in $I_m$ is an $(m, H)$-crossing.
The more general set-up to be used in \cite{BHST16b} will allow us to consider only paths $\sigma$ that coincide with the trajectory performed by the random walker.
\end{remark}

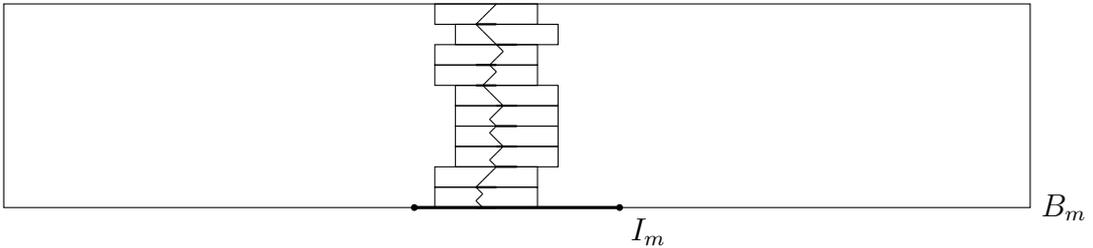
\begin{figure}[h]
  \centering
  \begin{tikzpicture}[scale=.9]
    \draw[very thick] (6,0) -- (9, 0); \node at (9, 0) [below right] {$I_m$};
    \fill (6,0) circle (0.05); \fill (9,0) circle (0.05);
    \xdef\z{7}
    \foreach \w in {0,...,9}
    { \draw[thin] ({.3 * floor(\z / .3) - .6},{.1 * (3 * \w)}) rectangle ({.3 * floor(\z / .3) + .9},{.1 * (3 * \w) + .3});
      \draw[thick] ({.3 * floor(\z / .3)},{.1 * (3 * \w)}) -- ({.3 * floor(\z / .3) + .3},{.1 * (3 * \w)});
      \foreach \y in {0,...,2}
      { \xdef\x{\z}
        \pgfmathparse{\x + .1 * 2 * (round((rand + 1)*0.7)-.5)}
        \xdef\z{\pgfmathresult}
        \draw[thin] (\x, {0.1 * (3 * \w + \y)}) -- (\z, {0.1 * (3 * \w + \y) + .1}) ;
      }
    }
    \draw (0,0) rectangle (15, 3); \node at (15, 0) [right] {$B_m$};
  \end{tikzpicture}
  \caption{The box $B_m$ for some $m \in M_{k+1}$ and an illustrative $m$-crossing $\sigma$.}
  \label{f:Bm_cross}
\end{figure}

The following definition plays a central role in our construction.
\begin{definition}\label{d:chi_g}
Given $g : \overline{\Omega} \to [-1,1]$, an index $m \in M_k$ and an $m$-crossing $\sigma: [n, \infty) \cap \mathbb{Z} \to \mathbb{Z}^d$, we define the average $\chi^g_\sigma$ of $g$ along $\sigma$ by
\begin{equation}\label{e:chi_g}
\chi^g_\sigma (\omega,U) := \frac{1}{L_k} \sum_{i=n}^{n + L_k - 1} g_{(\sigma(i), i)} (\omega,U),
\end{equation}
where $g_y = \theta_y g$ as in \eqref{e:def_gy}.
\end{definition}

Given a scale $\hat{k} \geq 0$ and $v_{\hat{k}} \in [L_{\hat{k}}^{-1/16}, 1]$,
recursively define (compare with \eqref{e:rhok})
\begin{equation}\label{e:vforchi}
v_{k+1} = v_k (1-L_k^{-1/16}), \quad k \ge \hat{k}.
\end{equation}
Note that $v_k$ decreases monotonically to (see \eqref{e:rho_star_star})
\begin{equation}\label{e:defvinfty}
v_\infty : = v_{\hat{k}} \prod_{k \ge \hat{k}} (1- L_k^{-1/16}) \in (0,1).
\end{equation}

Given $g: \overline{\Omega} \rightarrow [-1,1]$, $H : \overline{\Omega} \times \mathbb{Z}^d \to \{0,1\}$ and an integer $k \ge \hat{k}$,
define the events
\begin{equation}\label{e:A_m_with_chi}
A_m = \Big\{ \text{there exists an $(m, H)$-crossing $\sigma$ such that } \chi^g_\sigma < v_k \Big\}, \quad m \in M_k.
\end{equation}
Note that the events defined by \eqref{e:A_m_with_chi} are not necessarily adapted or monotone.
However, as already anticipated, we have the following.
\begin{lemma} \label{l:build_cascade}
The family $(A_m)_{m \in M_{\geq \hat{k}}}$ defined by \eqref{e:A_m_with_chi} is cascading.
\end{lemma}
\begin{proof}
Fix $k \ge \hat{k}$ and recall that we have assumed $v_{\hat{k}} \geq L_{\hat{k}}^{-1/16}$.
The first thing we note is that this inequality holds for all $k \geq \hat{k}$.
This indeed follows by induction using the definition of $v_k$ exactly as for \eqref{e:rhokdoesnotdecaytoofast}.

Next we claim that:
\begin{display}
\label{e:A_k+1_k}
if $A_m$ occurs for some $m \in M_{k+1}$, then there exist at least three elements\\
$m_i = (k, y_i)$ in $M_k$, $i=1,2,3$, with $y_i = (x_i,s_i) \in \mathbb{Z}^{d}\times \mathbb{Z}$ and $B_{m_i} \subset B_m$\\
satisfying $s_i \neq s_j$ when $i \neq j$ and such that $A_{m_i}$ occurs for $i = 1,2,3$.
\end{display}
Indeed, assume by contradiction that
\begin{display}
\label{e:at_most_two}
there are at most two elements $m' = (k,y')$, $m'' = (k, y'')$ in $M_k$
with $y' = (x',s')$ and $y'' = (x'',s'')$, $s' \neq s''$ and
$B_{m'}$, $B_{m''}$ contained in $B_m$\\ for which the events $A_{m'}$ and $A_{m''}$ happen.
\end{display}
Let $\sigma$ be an $(m, H)$-crossing.
We split its domain into disjoint intervals of length $L_k$:
\begin{equation}
\{n, \dots, n + L_{k+1} - 1\} = \bigcup_{j=1}^{J} \{n + (j-1)L_k, \dots, n + jL_k - 1\}, \text{ where $J = L_{k+1}/L_k$},
\end{equation}
Let us denote by $\sigma_j$, $j\in \{1, \ldots, J\}$, the restriction of $\sigma$ to $[n+(j-1)L_k, n+jL_k)$ which is again an $(m_j,H)$-crossing for an appropriate index $m_j$ in $M_k$, with $B_{m_j} \subset B_m$ (see Figure~\ref{f:Bm_cross}).
We now estimate
\begin{align}\label{e:displace}
\chi^g_\sigma
& \nonumber = \frac{1}{L_{k+1}} \sum_{i=n}^{n + L_{k+1} - 1} g_{(\sigma(i), i)} = \frac{1}{L_{k+1}} \sum_{j = 1}^J \sum_{i = n + jL_k}^{n + (j+1)L_k - 1} g_{(\sigma(i), i)} = \frac{L_k}{L_{k+1}} \sum_{j = 1}^J \chi^g_{\sigma_j}\\
& \overset{\eqref{e:at_most_two}}\geq -2 \frac{L_{k}}{L_{k+1}} + v_k \frac{L_k}{L_{k+1}} \left(\frac{L_{k+1}}{L_k} - 2\right)\\
& \nonumber \geq v_k - 4 \frac{L_k}{L_{k+1}} = v_{k+1} + (v_k-v_{k+1}) - 4 \frac{L_k}{L_{k+1}},
\end{align}
where, in the first inequality, we used the fact that,
if $A_{m_i}$ does not occur for some $i \in \{1,\ldots, J\}$, then $1 \geq \chi_\sigma^g \geq v_k$.
From the definition of $v_k$, we see that
\begin{equation}
\label{e:k2large}
v_k - v_{k+1} = v_k L_k^{-1/16} \geq \frac{1}{L^{1/8}_k} \geq \frac {4}{\lfloor L_k^{1/2} \rfloor} = 4\, \frac{L_k}{L_{k+1}},
\end{equation}
where for the second inequality we use $L_k \ge 10^{50}$ (cf.\ \eqref{e:Lk}).
Substituting this into \eqref{e:displace}, we get $\chi^g_\sigma \geq v_{k+1}$ so that $A_m$ cannot occur.
This proves the claim \eqref{e:A_k+1_k}.

Thus, on the event $A_m$, we may assume that there exist $m_1 = (k, y_1)$, $m_3 = (k, y_3)$ in $M_k$
where $y_1 = (x_1,s_1)$ and $y_3 = (x_3, s_3)$ with $s_3 \geq s_1 + 2$
(meaning that the vertical distance between $B_{m_3}$ and $B_{m_1}$ is at least $L_k$)
and such that both $A_{m_1}$ and $A_{m_3}$ occur.
This finishes the proof of the lemma.
\end{proof}

The events defined by \eqref{e:A_m_with_chi}
may be analysed with the help of Theorem~\ref{t:pk_decay}
whenever they are adapted and monotone.
We next give a complementary result stating that,
whenever the conclusion of Theorem~\ref{t:pk_decay} holds for $(A_m)_{m \in M_{\ge \hat{k}}}$,
it can be extended by interpolation to boxes of length $L \in \N$ (not necessarily of the form $L_k$).
We first need to extend the above definitions.

For $y\in\Z^d \times \Z$ and $L\in \N$, let
\begin{equation}
B_{y,L}=y+\{[-2 \mathfrak{R} L,3 \mathfrak{R} L)^d\times[0,L)\} \cap (\mathbb{Z}^{d}\times \mathbb{Z})
\end{equation}
and similarly
\begin{equation}
I_{y,L}=y+([0,\mathfrak{R} L)^d\times\lbrace 0\rbrace)\cap(\mathbb{Z}^{d}\times\mathbb{Z}).
\end{equation}
Given a function $H : \overline{\Omega} \times \mathbb{Z}^d \to \{0, 1\}$ and $y = (x, n)$, we define a $(y, L, H)$-crossing to be an $\mathfrak{R}$-Lipschitz function
$\sigma: [n, \infty) \cap \Z \to \Z^d$ such that $(\sigma(n), n) \in I_{y,L}$ and $h_\sigma(j) = 1$ for every $j \in [n, n + L) \cap \mathbb{Z}$.
If the function $H$ is identically equal to one we simply say that $\sigma$ is a $(y, L)$-crossing.

Finally, given $\sigma$ a $(y, L, H)$-crossing, we let
\begin{equation}\label{e:chi_g_L}
\chi^g_\sigma (\omega,U) := \frac{1}{L} \sum_{i=n}^{n + L - 1} g_{(\sigma(i), i)} (\omega,U).
\end{equation}
Our interpolation result reads as follows.
\begin{proposition}\label{p:interpolate}
Fix $\hat{k} \ge 0$, $v_{\hat{k}} \in [L_{\hat{k}}^{-1/16},1]$ 
and two functions $g: \overline{\Omega} \rightarrow[-1,1]$, $H : \overline{\Omega} \times \mathbb{Z}^d \to \{0, 1\}$.
Define $(A_m)_{m \in M_{\ge \hat{k}}}$ by \eqref{e:A_m_with_chi}.
Fix $\rho > 0$, $\gamma \in (1,3/2]$ and assume that
\begin{equation}\label{e:assump_interpolate}
p_k(\rho) = \max_{m \in M_k} \P^{\rho} \left( A_m \right) \le \exp \left\{ -(\log L_k)^{\gamma}\right\} \;\;\; \forall \; k \ge \hat{k}.
\end{equation}
Then, for every $\varepsilon > 0$, there exists $c > 0$ such that
\begin{equation}\label{e:interpolate}
\P^\rho\left(\exists \text{ a }(0,L)\text{-crossing }\sigma\text{ s.t. }\chi_\sigma^g< v_{\infty}-\varepsilon \right)
\le c^{-1} e^{- c (\log L)^\gamma }
\;\;\;\; \forall\; L \ge 1,
\end{equation}
where $v_\infty$ is given by \eqref{e:defvinfty}.
\end{proposition}
\begin{proof}
We follow the proof of Lemma~3.5 in \cite{HHSST14}.
We may assume $L$ to be so large that, defining $\check{k}$ by $2L_{\check{k}+2}\leq L<2L_{\check{k}+3}$,
then $\check{k}\geq \hat{k}$ and $\frac{L_{k+1}}{L_{k}}>1+2/\varepsilon$ for all $k\geq \check{k}$.

We first consider multiples of $L_k$, $k\geq\check{k}$. Define
\begin{equation}
M'_k:=\lbrace m\in M_k : B_m\subset B_{(k+2, 0)}\rbrace,
\end{equation}
\begin{equation}
\mathcal{B}_{\check{k}}:=\bigcap_{k\geq\check{k}}\bigcap_{m\in M'_k}A_m^c
\end{equation}
\begin{equation}
\text{and}\qquad J_{\check{k}}:=\bigcup_{k\geq \check{k}}\bigcup_{l=1}^{L_{k+2}/L_k}\lbrace l L_k\rbrace\subset \N.
\end{equation}
Let us see that $\mathcal{B}_{\check{k}}$ has high probability. Indeed,
\begin{align}\label{e:printerp1}
\P^\rho\left(\mathcal{B}_{\check{k}}^c \right) \leq \sum_{k\geq \check{k}}\sum_{m\in M'_k}\mathbb{P}(A_m)
& \overset{\text{\eqref{e:assump_interpolate}}}{\leq}
\sum_{k\geq \check{k}} c \left(\frac{L_{k+2}}{L_k}\right)^{d+1}e^{- (\log L_k)^\gamma}
\nonumber\\
& \;\; \leq  \;
c^{-1} e^{- c(\log L_{\check{k}})^\gamma} \le c^{-1} e^{- c (\log L)^{\gamma}},
\end{align}
where for the last inequality we used that $L_{\check{k}} \ge L_{\check{k}+3}^{(2/3)^3} > \tfrac12 L^{(2/3)^3}$.

We now claim that:
\begin{equation}\label{e:LinJ}
\text{on $\mathcal{B}_{\check{k}}$, for any $L' \in J_{\check{k}}$ and any $((0,L'),h)$-crossing $\sigma$, we have $\chi_\sigma^g\geq v_\infty$.}
\end{equation}
Let us prove this for fixed $k\geq \check{k}$ by induction on $l\leq L_{k+2}/L_k$.
For $l=1$, the claim holds since any $((0, L_k), h)$-crossing $\sigma$ is indeed an $(m, H)$-crossing for $m=(k, 0)$ so that, on the event $A_m^c \subset \mathcal{B}_{\check{k}}$, we have $\chi_\sigma^g \geq v_k \geq v_\infty$.
Assume that it holds for $L'=l L_k$ with $1\leq l<L_{k+2}/L_k$.
Fix a $((0,(l+1)L_k),h)$-crossing $\sigma$.
Notice that there exists $m\in M'_k$ such that $(\sigma(lL_k),lL_k)\in I_m$.
Therefore, $\sigma$ restricted to times $i=lL_k,\ldots, (l+1)L_k-1$ is an $(m,H)$-crossing and $\sigma$ restricted to times $i=0,\ldots, lL_k-1$ is a $((0,l L_k),h)$-crossing.
Hence, by induction and since we are on $\mathcal{B}_{\check{k}}$,
\begin{eqnarray}
\sum_{i=0}^{(l+1)L_k -1}g_{(\sigma(i),i)} & = & \sum_{i=0}^{lL_k -1}g_{(\sigma(i),i)}+\sum_{i=lL_k}^{(l+1)L_k -1}g_{(\sigma(i),i)}\nonumber\\
&\geq& v_\infty lL_k+ v_k L_k \ge v_\infty (l+1)L_k,
\end{eqnarray}
completing the induction step and proving \eqref{e:LinJ}.

We now prove \eqref{e:interpolate}. In fact, we show that
\begin{equation}\label{e:printerp_end}
\text{on $\mathcal{B}_{\check{k}}$, $\chi_\sigma^g\geq v_{\infty} -\varepsilon$ for any $(0,n)$-crossing $\sigma$ with $n \geq 2L_{\check{k}+2}$}.
\end{equation}
To this end, for $n$ as in \eqref{e:printerp_end}, define $\bar{k}=\bar{k}(n)$ to be the smallest integer
such that $n\in[ \bar{l}L_{\bar{k}},(\bar{l}+1)L_{\bar{k}})$ for some $1 \le \bar{l}\leq L_{\bar{k}+2}/L_{\bar{k}} $.
Note that $\bar{k} \ge \check{k}$ and that, by the minimality of $\bar{k}$,
we have $n \ge L_{\bar{k}+1}$ and thus $\bar{l}\geq 2/\varepsilon$.
Then estimate
\begin{eqnarray}
\sum_{i=0}^{n -1}g_{(\sigma(i),i)}&=&\sum_{i=0}^{\bar{l}L_{\bar{k}}-1}g_{(\sigma(i),i)}+\sum_{i=\bar{l}L_{\bar{k}}}^{n -1}g_{(\sigma(i),i)}\nonumber\\
&\geq & v_\infty \bar{l}L_{\bar{k}}-L_{\bar{k}} = L_{\bar{k}}((v_\infty-\varepsilon)\bar{l}+\varepsilon\bar{l}-1) \nonumber\\
&\geq& L_{\bar{k}}(v_\infty- \varepsilon)(\bar{l}+1) \geq n(v_{\infty}-\varepsilon),
\end{eqnarray}
where we used \eqref{e:LinJ}, $|g| \le 1$ and $v_\infty \le 1$.
This finishes the proof.
\end{proof}

We may now state our target corollary,
which conveniently summarizes ballisticity-type results
without explicit reference to most of the technical renormalisation setup. Recall \eqref{e:chi_g_L} and the definition above.

\begin{corollary}\label{cor:renormnoscales}
Let $(L_k)_{k \in Z_+}$ be given by \eqref{e:Lk}.
For any $\gamma \in (1,3/2]$, there exists an index $k_o = k_o(\gamma, d) \in \N$ satisfying the following.
Fix two functions $g :\overline{\Omega} \to [-1, 1]$, $H : \overline{\Omega} \times \mathbb{Z}^d \to \{0, 1\}$
and two non-negative sequences $v(L), \rho(L)$.
Assume that, for some $L_* \in \N$ and all $L \ge L_*$,
$v(L) \wedge \rho(L) \geq L^{-1/16}$ and, for any $\hat{v} >0$, the event
\begin{equation}
\label{e:H_g_non_decr}
\big\{\text{there exists a $(0, L, H)$-crossing $\sigma$ such that $\chi_\sigma^g  \leq \hat{v}$}\big\}
\end{equation}
is measurable in $\sigma(N(y), U_y \colon\, y \in B_{0,L})$ and non-increasing (respectively non-decreasing).
Assume additionally that, for some $\hat{k} \ge k_o$ such that $L_{\hat{k}} \ge L_*$,
\begin{equation}
\label{e:H_g_upper_bound}
\mathbb{P}^{\rho(L_{\hat{k}})}(\exists \text{ a $(0, L_{\hat{k}}, H)$-crossing $\sigma$ such that $\chi_\sigma^g \leq v(L_{\hat{k}})$}) \leq \exp(-(\log L_{\hat{k}})^\gamma).
\end{equation}
Then there exist (explicit) $\rho_\infty, v_\infty>0$ such that, for each $\varepsilon>0$, 
\begin{equation}
\mathbb{P}^{\rho}( \exists \text{ a $(0, L, H)$-crossing $\sigma$ such that $\chi_\sigma^g \leq v_\infty - \varepsilon$}) \leq c^{-1}\exp(-c (\log L)^\gamma)
\end{equation}
for some $c \in (0,\infty)$, all $L \in \N$ and all $\rho \ge \rho_\infty$ (respectively, all $\rho \le \rho_\infty$).
\end{corollary}

Before we proceed to the proof, a few words about Corollary~\ref{cor:renormnoscales}.
Assumption \eqref{e:H_g_upper_bound} can be interpreted as a \emph{triggering condition},
i.e., an a-priori estimate that must be provided in order to start the renormalisation procedure.
The measurability and monotonicity assumptions must be checked in each case.
Note that measurability follows whenever $g$ and $H(\cdot,x)$ (for all $x \in \Z^d$) are \emph{local} 
(i.e., supported in a finite set in the sense of Definition~\ref{e:defsupportedK}) 
and \emph{instantaneous}, where we say that a function $f:\overline{\Omega} \to \R$ is instantaneous if
$f(\omega, U) \in \sigma(N(z,0),U_{(z,0)} \colon\, z \in \Z^d)$,
i.e., $f$ depends only on one time slice of the random environment.

\begin{proof}[Proof of Corollary~\ref{cor:renormnoscales}]
Let $k_o \in \N$ be as in the statement of Theorem~\ref{t:pk_decay},
and fix $\hat{k} \ge k_o$ satisfying $L_{\hat{k}} \ge L_*$ and \eqref{e:H_g_upper_bound}.
Setting $v_{\hat{k}} := v(L_{\hat{k}})$, define $v_k$, $k > \hat{k}$ as in \eqref{e:vforchi}
and $v_\infty$ as in \eqref{e:defvinfty}.
For $k \geq \hat{k}$, let $A_m$ be defined as in \eqref{e:A_m_with_chi} and $p_k(\rho)$ as in \eqref{e:pk}.
Note that the events $A_m$, for $m \in M_{\geq \hat{k}}$ as above are cascading, adapted and non-decreasing (resp.\ non-increasing) according to Lemma~\ref{l:build_cascade} and our assumptions.

Set now $\hat{\rho} := \rho(L_{\hat{k}})$ and note that, by \eqref{e:H_g_upper_bound},
\begin{equation}
p_{\hat{k}}(\hat{\rho}) \leq \exp \big( - (\log  L_{\hat{k}})^\gamma \big).
\end{equation}
Therefore we can use Theorem~\ref{t:pk_decay} to conclude that, for some $\rho_\infty > 0$ (more precisely, $\rho_\infty := \rho_*$ in the non-increasing case, or $\rho_\infty := \rho_{**}$ in the non-decreasing case),
\begin{equation}
p_k(\rho) \leq \exp \big\{ - (\log L_k)^{\gamma} \big\}, \text{ for every $k \geq \hat{k}$}
\end{equation}
for any $\rho \geq \rho_\infty$ in the non-increasing case, or any $\rho \le \rho_\infty$ in the non-decreasing case. 
The conclusion then follows from Proposition~\ref{p:interpolate}.
\end{proof}

\section{Applications}
\label{s:applications}

This section is dedicated to applying the renormalization setup developed in Section~\ref{s:renormalization} to show ballistic behavior of two processes.
Namely, for a random walker in the environment of simple random walks and for the front of an infection process.

\subsection{Random walker on random walks (large density)}
In this subsection, we will prove a ballisticity result for the random walker in the environment of simple random walks,
generalizing Theorem~1.5 of \cite{HHSST14}.
Let
\begin{equation}\label{e:defcH}
\cH_{v,L} := \left\{(x,n) \in \Z^d \times \Z \colon\, x \cdot e_1 \le -L + v n \right\}.
\end{equation}
\begin{theorem}[Ballisticity condition]
\label{thm:ballisticity}
For any $v_\star \in (0,v_\bullet)$,
there exists $\rho_\star \in (0,\infty)$ and $c>0$ such that, for all $\rho \ge \rho_\star$,
\begin{equation}\label{e:LD}
\P^\rho \left( \exists \, n \ge 0 \colon\, Y_n \in \cH_{v_\star, L}  \right) \le c^{-1} e^{-c (\log L)^{3/2}} \;\;\; \forall \; L \ge 1.
\end{equation}
\end{theorem}
Theorem~\ref{thm:ballisticity} will be proved by means of two propositions stated and proved below.
Both the theorem and these intermediate results will be crucial to control
the tail of the regeneration time constructed in Section~\ref{s:regmanyRWshighdim}.

The next proposition is very intuitive, stating that if the density is high enough then all paths stay most of their time on points with a large number of particles.

%\newconstant{c:stay_on_full}
%\newconstant{c:stay_on_full_2}
\begin{proposition}[Uniform density control along paths]
  \label{prop:enoughparticles}
  For all $K \in \N$ and $\varepsilon \in (0,1)$, there exists $c > 0$ and $\rho(K,\varepsilon) \in (0,\infty)$ such that, for all $\rho \ge \rho(K, \varepsilon)$,
\begin{equation}\label{e:enoughparticles}
\P^\rho \Big( \exists\, \ell \ge L/(2 \mathfrak{R}),  \sigma \in \mathfrak{S} \colon\, \sum_{i=0}^{\ell-1} \mathbbm{1}_{\{N(\sigma(i), i) \ge K \}} < (1-\varepsilon) \ell \Big) \le c^{-1} e^{-c (\log L)^{3/2}}
  \end{equation}
  for all $L \geq 1$.
\end{proposition}
\begin{proof}
Take $k_o$ as in the statement of Theorem~\ref{t:pk_decay} for $\gamma = 3/2$,
and let $\hat{k} \ge k_o$ be large enough such that $\prod_{k\geq \hat{k}}(1-L_k^{-1/16})\geq 1-\varepsilon/2$.
Choose $v_{\hat{k}}:=1$, $g: =\mathbbm{1}_{\{N(0,0) \ge K \}}$, $H\equiv 1$ (thus, we will say only $m$-crossing instead of $(m,H)$-crossing).
Define the family $(A_m)_{m \in M_{\geq \hat{k}}}$ as in \eqref{e:A_m_with_chi} and note that it is adapted and that each $A_m$ is non-increasing.
For a fixed $\hat{\rho}>0$, consider the crude bound
\begin{eqnarray}\label{e:prpropenoughpart1}
p_{\hat{k}}(\hat{\rho})&\leq &\P^{\hat{\rho}}\left(\exists (y,n)\in B_{L_{\hat{k}}} \text{ such that }N(y,n)<K\right) \nonumber\\
&\leq& (5 \mathfrak{R} L_{\hat{k}})^dL_{\hat{k}} \P^{\hat{\rho}}(N(0,0)<K)\leq  (5 \mathfrak{R} L_{\hat{k}})^dL_{\hat{k}} K (\hat{\rho} \vee 1)^K e^{-\hat{\rho}}.
\end{eqnarray}
For fixed $K,\hat{k}$, we can choose $\hat{\rho} \ge L^{-1/16}_{\hat{k}}$ such that the right-hand side of \eqref{e:prpropenoughpart1}
is less than $\exp(-(\log L_{\hat{k}})^{3/2})$.
Therefore, by Theorem~\ref{t:pk_decay} and Proposition~\ref{p:interpolate},
there exists $c>0$ such that, for all $\rho \ge \rho(K,\varepsilon) := \iota_{\hat{k}} \hat{\rho}$ (with $\iota_{\hat{k}}$ as in \eqref{e:defiota}) and all $\ell \ge 1$,
\begin{equation*}
\P^\rho \Big( \text{there exists a $(0,\ell)$-crossing $\sigma$ with } \sum_{i=0}^{\ell-1} \mathbbm{1}_{\{N(\sigma(i), i) \ge K \}} < (1-\varepsilon) \ell \Big)\leq c^{-1} e^{-c(\log \ell)^{3/2}}
\end{equation*}
The proposition follows by noticing that the first $\ell$ steps of any $\sigma\in\mathfrak{S}$ form a $(0,\ell)$-crossing,
and then applying a union bound.
\end{proof}

Our second proposition is a quenched deviation estimate for the position of the random walk.
Intuitively speaking, it says that if all paths spend a large proportion of their time in sites with many particles, then the random walker itself has to move ballistically.

For technical reasons we first have to restrict our attention to the collection of paths that behave well in a certain sense.
For $L \in \N$ and $v \in (0,\mathfrak{R}]$, let $\mathfrak{S}^{v, L}$ be those paths of $\mathfrak{S}$ that never touch $\cH_{v, L}$.
More precisely
\begin{equation}\label{e:deffrakSLv}
\mathfrak{S}^{v,L} := \left\{ \sigma \in \mathfrak{S} \colon\, (\sigma(i),i) \notin \cH_{v, L} \, \forall \, i \in \Z_+ \right\}.
\end{equation}
For $K \in \N$, $\varepsilon>0$ and $y \in \Z^d \times \Z$, let
\begin{equation}\label{e:defcA}
\cA^{L, v, K, \varepsilon}_y := \bigg\{ \exists \, \ell \ge L/(2 \mathfrak{R}), \sigma \in \mathfrak{S}^{v, L} \colon\, \sum_{i=0}^{\ell-1} \mathbbm{1}_{\{N\left(y+(\sigma(i),i)\right) \ge K \}} < (1-\varepsilon) \ell \bigg\}.
\end{equation}
\begin{proposition}[Quenched deviation estimate]
\label{prop:quencheddev}
For all $v_\star \in (0,v_\bullet)$, there exist $k_\star \in \N$,
$\varepsilon_\star \in (0,1)$ and $c > 0$ such that,
for all $\rho \in (0,\infty)$, $y \in \Z^d \times \Z$, $v \le v_\star$ and $K \ge k_\star$,
$\P^\rho$-almost surely, if the event $\cA^{L, v, K,\varepsilon_\star}_y$ does not occur then
\begin{equation}\label{e:quencheddev}
\P^\rho \big( \exists \, \ell \ge 0 \colon\, Y^y_\ell - y \in \cH_{v, L}  \,\big|\, \omega \big) \le c^{-1} e^{-c L}.
\end{equation}
\end{proposition}
\begin{proof}
Let $y = (x,n) \in \Z^d \times \Z$ be fixed.
Fix $\delta_\star \in (0,1)$ satisfying $v_\star+ 2 \delta_\star < v_\bullet$.
By Assumption (D), there exists a $k_\star \in \N$ such that
\begin{equation}\label{e:largedriftafterkstar}
\inf_{k \ge k_\star} \sum_{z \in \cS} \alpha(k,z) z \cdot e_1 > v_\star + 2 \delta_\star.
\end{equation}
Take $\varepsilon_\star \in (0,1)$ small enough such that $2 (\mathfrak{R} +1) \varepsilon_\star < \delta_\star$,
and fix $K\ge k_\star$, $v \le v_\star$.

For $(z,l) \in \Z^d \times \Z$, let
\begin{equation}\label{e:deflocaldrift}
d^\omega(z,l) := \E^\rho \left[ (X^{(z,l)}_1 - z)\cdot e_1 \;\middle| \; \omega \right] = \sum_{u \in \cS} u \cdot e_1 \, \alpha(N(z,l),u)
\end{equation}
denote the quenched local drift in direction $e_1$ at the point $(z,l)$.
For $\sigma \in \mathfrak{S}$, let
\begin{equation}\label{e:deftotaldrift}
D^\omega_\ell(\sigma) := \sum_{k=0}^{\ell-1} d^\omega(y+(\sigma(k),k))
\end{equation}
be the total drift accumulated along the path that starts at $y$ and has increments given by $\sigma$ up to time $\ell \in \N$.
When $\sigma = X^y-x$ we omit it and write $D^\omega_\ell$.
On $(\cA_y^{L,v,K,\varepsilon_\star})^c$, for all $\ell \ge L/(2 \mathfrak{R})$ and all $\sigma \in \mathfrak{S}^{v,L}$,
\begin{align}\label{e:propdrift}
D^{\omega}_\ell(\sigma)
& \ge (v_\star + 2 \delta_\star) \sum_{k=0}^{\ell-1} \mathbbm{1}_{\{N(y+(\sigma(k),k)) \ge K \}} - \mathfrak{R} \sum_{k=0}^{\ell-1} \mathbbm{1}_{\{N(y+(\sigma(k),k)) < K \}} \nonumber\\
& \ge \left[v_\star + 2 \delta_\star - \varepsilon_\star (v_\star + \mathfrak{R} + 2 \delta_\star) \right] \ell \nonumber\\
& > (v_\star+\delta_\star) \ell \ge (v+\delta_\star) \ell.
\end{align}
by our choice of $\delta_\star$, $k_\star$ and $\varepsilon_\star$.

Note that, under $\P^\rho(\cdot \,| \omega)$, the process
\begin{equation}\label{e:defmartingale}
\begin{aligned}
M_\ell & := (X^y_\ell -x) \cdot e_1 - D^\omega_\ell \\
& = \sum_{k=0}^{\ell-1} (X^y_{k+1}-X^y_k) \cdot e_1 - \E^\rho \left[ (X^y_{k+1}-X^y_k) \cdot e_1 \;\middle|\; X^y_0, \ldots, X^y_k, \omega \right]
\end{aligned}
\end{equation}
is a zero-mean martingale with respect to the filtration $\sigma(X^y_0, \ldots, X^y_\ell)$,
and has increments bounded by $2 \mathfrak{R}$.
Therefore, by Azuma's inequality and a union bound, there exists a $c>0$ such that
\begin{equation}\label{e:applyAzuma}
\P^\rho \left( \exists\, \ell \ge L/(2 \mathfrak{R}) \colon\, |M_\ell| \ge \delta_\star \ell \;\middle|\; \omega \right) \le c^{-1} e^{-c L} \quad \forall \; L \in \N.
\end{equation}

Now we argue that, on $(\cA_y^{L,v, K, \varepsilon_\star})^c$,
\begin{equation}\label{e:lastargquencheddev}
\left\{\exists \, \ell \ge 0 \colon\, Y^y_\ell - y \in \cH_{v, L} \right\} \subset \left\{\exists \, \ell \ge L/(2 \mathfrak{R}) \colon\, |M_\ell| \ge \delta_\star \ell \right\}.
\end{equation}
Indeed, let $\ell_0 \in \N$ be the smallest time satisfying $Y^y_{\ell_0} - y \in \cH_{v, L}$. Then $\ell_0 \ge L/(
\mathfrak{R} + v) \ge L/(2\mathfrak{R})$.
Setting $\sigma = X^y - x$ up to time $\ell_0-1$ and equal to an arbitrary $\mathfrak{R}$-Lipschitz path that does not touch $\cH_{v,L}$ for times greater than $\ell_0$, then $\sigma \in \mathfrak{S}^{v, L}$ and we obtain by \eqref{e:propdrift} that, on $(\cA_y^{L,v, K, \varepsilon_\star})^c$,
$D^\omega_{\ell_0} \ge (v + \delta_\star) \ell_0$.
If additionally $|M_{\ell_0}| < \delta_\star \ell_0$ would hold, then we would have a contradiction since
\begin{equation}\label{e:contradiction}
(X^y_{\ell_0}-x) \cdot e_1 \ge D^\omega_{\ell_0}(\sigma) - |M_{\ell_0}| > v \ell_0 \;\; \Rightarrow \;\; Y^y_{\ell_0} - y \notin \cH_{v, L}.
\end{equation}
This shows \eqref{e:lastargquencheddev}, and the conclusion follows by \eqref{e:applyAzuma}.
\end{proof}

Propositions~\ref{prop:enoughparticles}--\ref{prop:quencheddev} imply the ballisticity condition \eqref{e:LD} as follows.
\begin{proof}[Proof of Theorem~\ref{thm:ballisticity}]
For $v_\star \in (0, v_\bullet)$, fix $k_\star \in \N$, $\varepsilon_\star>0$ as in Proposition~\ref{prop:quencheddev}
and set $\rho_\star := \rho(k_\star, \varepsilon_\star)$ as in Proposition~\ref{prop:enoughparticles}.
The theorem follows by noting that
\begin{equation}\label{e:proofthmballisticity}
\cA_y^{L,v,K,\varepsilon} \subset \left\{ \exists\, \ell \ge L/(2 \mathfrak{R}), \sigma \in \mathfrak{S} \colon\,
\sum_{k=1}^{\ell-1} \mathbbm{1}_{\{N(y+ (\sigma(k),k) ) \ge K \}} < (1-\varepsilon) \ell \right\}
\end{equation}
and that the probability of the right-hand side of \eqref{e:proofthmballisticity} does not depend on $y$.
\end{proof}

%%%%%%%%%%%%%%%%%%%%%%%%%%%%%%%%%%%%%%%%%%%%%%%%%%%%%%%%%%%%%%

\subsection{Infection}
\label{ss:infection}

In this subsection,
we prove Proposition~\ref{p:infection} regarding the front of the infection process described in the introduction.
We start with a precise construction of the model.

Fix $\rho > 0$, $d=1$ and let $N(z,0)$ and $S^{z,i}$ be as in Section~\ref{s:construction}, i.e., $(N(z,0))_{z \in \Z}$ are i.i.d.\ Poisson($\rho$)
random variables and $(S^{z,i}-z)_{z \in \Z, i \in \N}$ are i.i.d., each distributed as a double-sided simple symmetric random walk on $\Z$ started at $0$.

We also introduce random variables $\eta(z,i,n) \in \{0,1\}$ to indicate whether the particle corresponding to $S^{z,i}$
is healthy ($\eta(z,i,n)=0$) or infected ($\eta(z,i,n)=1$) at time $n$.
We will define them recursively as follows.
Set the initial configuration to be
\begin{align}
  & \eta(z,i,0) = 1 \quad \text{ if } z \leq 0 \text{ and } i \leq N(z,0),\\
  & \eta(z,i,0) = 0 \quad \text{ otherwise.}
\end{align}
Supposing that, for some $n \ge 0$, $\eta(z,i,n)$ is defined for all $z \in \mathbb{Z}$, $i \in \N$, we set
\begin{equation}\label{e:defeta}
\eta(z,i,n+1) = \left\{
\begin{array}{ll}
1 & \begin{array}{l}
\text{if } i \le N(z,0) \text{ and there exists} \\  z' \in \Z, i' \in \N \text{ with } \eta(z',i',n) = 1, S^{z',i'}_n = S^{z,i}_n,
 \end{array}\\
0 & \text{ otherwise.}
\end{array}\right.
\end{equation}
This definition means that,
whenever a collection of particles share the same site at time $n$,
if one of them is infected then they will all become infected at time $n + 1$.

We are interested in the process $\bar{X} = (\bar{X}_n)_{n \in \Z_+}$ defined by
\begin{equation}\label{e:defbarX}
  \bar{X}_n = \max\{S^{z,i}_n \colon\, \eta(z,i,n) = 1\},
\end{equation}
i.e., $\bar{X}_n$ is the rightmost infected particle at time $n$.
We call $\bar{X}$ the \emph{front of the infection}.

Note that the process $\eta$ differs slightly from that described in the introduction,
where particles sharing a site with an infected one were required to become immediately infected.
However, it is easy to check that the process $\bar{X}$ is not affected by this difference,
and we choose to work with $\eta$ for simplicity.

Our first result towards Proposition~\ref{p:infection} is a reduction step,
stating that it suffices to find, with high probability, enough times $n$ when the front
$\bar{X}_n$ of the infection process is close to another particle.
For this we fix $r \geq 0$ and define $g_r$ by
\begin{equation}\label{e:defgr}
  g_r = \mathbbm{1}_{ \big\{ \textstyle{\sum_{x \in [-r,r] \cap (2 \Z)}} N(x,0) \ge 2 \big\}},
\end{equation}
that is, $g_r$ is the indicator function of the event that, at time zero,
there are at least two particles at even sites within distance $r$ from the origin.
Our lemma reads as follows.

\begin{lemma}
\label{l:infection_reduction}
Fix $\rho > 0$ and $r \geq 0$ and suppose that, for some $h \in (0,1)$, $c > 0$,
\begin{equation}
\label{e:always_close}
\mathbb{P}^{\rho}
\Big(
\begin{array}{c}
\chi^{g_r}_\sigma \geq h \text{ for every $1$-Lipschitz path}\\
\text{$\sigma:\{0,\dots,L\} \to \mathbb{Z}$ with $|\sigma(0)| < L$}
\end{array}
\Big) \geq 1 - c^{-1} e^{ -c (\log L)^{3/2} } \;\;\; \forall\, L \ge 1,
\end{equation}
where $g_r$ is as in \eqref{e:defgr} and $\chi^{g_r}_\sigma$ as in \eqref{e:chi_g}.
Then \eqref{e:infection} holds for some some $v, c > 0$.
\end{lemma}
\begin{proof}
One can check from the definition of the rightmost infected particle that the increment
$\bar{X}_{n+1} - \bar{X}_n$ always dominates that of a symmetric random walk on $\Z$.
At some steps, however, this increment has a drift to the right, namely when there is more than one particle at $\bar{X}_n$.
The idea of the proof will be to bound the number of times at which such positive drift is observed.

We first note that the front starts close to the origin.
Indeed,
\begin{equation}\label{e:controlfrontat0}
\P^\rho \left( \bar{X}_0 < - \sqrt{L} \right) = \P^{\rho} \left( N(z,0) = 0 \;\forall\; z \in [-\sqrt{L}, 0] \cap \Z \right) \le c e^{-\rho \sqrt{L}}.
\end{equation}

Now, at every time $n'$ at which there is another particle at distance at most $r$ from the front $\bar{X}_{n'}$
at a site with the same parity as $\bar{X}_{n'}$,
we can use the Markov property to see that, with uniformly positive probability,
this additional particle will reach the front within the next $r$ steps.
This means that, if $n'$ is such a time, the increment $\bar{X}_{n' + r + 1} - \bar{X}_{n'}$ stochastically dominates
(under the conditional law given $( N(\cdot, \ell))_{\ell \le n'}$) a random variable $\zeta$
with positive expectation satisfying $|\zeta| \leq r+1$.
We will show that $v = h E[\zeta]/(3(r+1))$ fulfills \eqref{e:infection}.

Consider the $1$-Lipschitz path given by the front $(\bar{X}_{\ell})_{\ell = 0}^{L}$.
Denote by $\mathcal{D}$ the intersection of the event appearing in \eqref{e:always_close} with $\{\bar{X}_0 \ge - \sqrt{L}\}$.
On $\mathcal{D}$, we see that, for at least $\lfloor h L \rfloor$ steps between times zero and $L$,
the front $\bar{X}$ is $r$-close to another particle.
Therefore, the same happens for at least $k_L := \lfloor \lfloor h L \rfloor /(r+1) \rfloor$
steps that are at least $r + 1$ time units apart from each other,
and we can estimate using the Markov property
\begin{equation*}
\P^{\rho}[\bar{X}_{L} < vL] \leq \P^{\rho}(\mathcal{D}^c) + \P^\rho (\zeta_1 + \dots + \zeta_{k_L} < 2v L + \sqrt{L}) + \P^{\rho}(S^{0,1}_{L - (r+1) k_L} < -v L),
\end{equation*}
where the $\zeta_i$'s are i.i.d.\ and distributed as $\zeta$.
Applying standard large deviation estimates to the sum of the $\zeta_i$'s and to $S^{0,1}$, we see that
\begin{equation}
\mathbb{P}^{\rho}(\bar{X}_n < v L) \leq \mathbb{P}^{\rho}(\mathcal{D}^c) + c^{-1} \exp\{-c L\} +  \leq c^{-1} \exp \{-c (\log L)^{3/2}\},
\end{equation}
finishing the proof of the lemma.
\end{proof}

We next present the proof of Proposition~\ref{p:infection}.
In light of Lemma~\ref{l:infection_reduction}, all we need to prove is \eqref{e:always_close},
and for this we will use the renormalization procedure developed in Section~\ref{s:renormalization}.
One might try to obtain \eqref{e:always_close} by direct application of Theorem~\ref{t:pk_decay},
defining the events $A_m$ in a natural way and then taking $r$ large enough.
There is however a serious problem with this approach:
for large values of $r$, the family $A_m$ will no longer be adapted in the sense of Definition~\ref{d:adapted}.
To circumvent this issue, we define intermediate classes of events that will certainly be adapted,
although not necessarily cascading.
The details are carried out next.
\begin{proof}[Proof of Proposition~\ref{p:infection}]
Given $\rho > 0$, let $\hat{\rho} := \iota_0^{-1} \rho$ (cf.\ \eqref{e:defiota}).
Fix $h_0 = 1/2$ and define inductively the sequence $h_k$ by $h_{k+1} = h_k(1 - L_k^{-1/16})$.
Since $h_0 \geq L_0^{-1/16}$, using \eqref{e:Lk} it follows by induction that $h_k \geq L_k^{-1/16}$ for every $k \geq 0$.
Moreover, $h_k$ decreases monotonically to $h_\infty := (2 \iota_0)^{-1} > 0$.
For $m \in M_k$, let
\begin{equation}
A'_m = \Big\{\text{there exists an $m$-crossing $\sigma$ such that $\chi^{g_{L_k}}_\sigma < h_k$} \Big\}.
\end{equation}
In the definition of $A'_m$ we have used the local function $g_{L_k}$,
which means that we are looking for particles on even sites at distance at most $L_k$ from the origin.
Intuitively speaking, this task will become easier and easier to accomplish  as $k$ grows.
This is made precise in the following claim: there exists a $c>0$ such that
\begin{equation}\label{e:A_prime_m_decay}
\P^{\hat{\rho}}(A'_m) \leq c^{-1} \exp \{-c L_k\}, \text{ for every $m \in M_{k}$, $k \ge 0$}.
\end{equation}
Indeed, this follows from a union bound over the points of the box $B_m$ together
with a simple large deviations estimate on the sum of independent Poisson($\hat{\rho}$) random variables.

By \eqref{e:A_prime_m_decay}, there exists a $\hat{k}_o \in \N$ such that
\begin{equation}\label{e:A_prime_m_decay2}
\P^{\hat{\rho}}(A'_m) \leq \exp \{-( \log L_k)^{3/2}\}, \text{ for every $m \in M_{k}$, $k \ge \hat{k}_o$}.
\end{equation}
As mentioned above, the family $A'_m$ may not be cascading, however it is clearly adapted.
We now define another collection that will indeed be cascading.
Let $k_o$ as in the statement of Theorem~\ref{t:pk_decay}
and take $\hat{k} \geq k_o \vee \hat{k}_o$.
Then define, for $k \ge \hat{k}$ and $m \in M_k \subset M_{\geq \hat{k}}$,
\begin{equation}\label{e:A_m_notprime}
A_m = \Big\{\text{there exists an $m$-crossing $\sigma$, such that $\chi^{g_{L_{\hat{k}}}}_\sigma < h_k$} \Big\}.
\end{equation}
Note here that the local function $g_{L_{\hat{k}}}$ is fixed, i.e.\ it does not depend on the scale $k$ associated with $m$.
This allows us to employ Lemma~\ref{l:build_cascade} and conclude that
\begin{display}\label{e:A_m_cascading}
the family $(A_m)_{m \in M_{\geq \hat{k}}}$ is cascading.
\end{display}
Moreover, this collection is adapted and composed of non-increasing events.

In view of \eqref{e:A_prime_m_decay2} and \eqref{e:A_m_cascading},
we have $p_{\hat{k}}(\hat{\rho}):= \sup_{m \in M_{\hat{k}}} \P^{\hat{\rho}} \left( A_m \right) \leq \exp\{ -(\log L_{\hat{k}})^{3/2} \}$
since $A'_m  = A_m$ for $m \in M_{\hat{k}}$.
Applying Theorem~\ref{t:pk_decay} and Proposition~\ref{p:interpolate},
we obtain, since $\rho = \iota_0 \hat{\rho} \ge \iota_{\hat{k}} \hat{\rho}$ and by translation invariance,
\begin{equation}
  \label{e:prinfect_last}
  \begin{split}
    \P^\rho & \bigg(
    \begin{array}{c}
      \text{there exists a $1$-Lipschitz path } \sigma:[0,L) \cap \Z \to \Z \text{ that is either}\\
      \text{a } (0,L) \text{-crossing or a } ((-L,0), L) \text{-crossing with } \chi_\sigma^{g_{L_{\hat{k}}}} <  h_\infty/2
    \end{array}
    \bigg)\\
    & \le c^{-1} e^{- c(\log L)^{3/2}},
\end{split}
\end{equation}
implying \eqref{e:always_close}.
Proposition~\ref{p:infection} then follows from Lemma~\ref{l:infection_reduction}.
\end{proof}

%%%%% End of text for renormalizationmany %%%%%

%%%%%%\subfile{regmany}%%%%%%%
%%%%%% Subfile replaced by text below%%%%%
\section{Regeneration: proof of Theorem~\ref{thm:main}}
\label{s:regmanyRWshighdim}\noindent
In this section, we adapt Section 4 of \cite{HHSST14} to our setting using Propositions~\ref{prop:enoughparticles} and \ref{prop:quencheddev}.
Theorem \ref{thm:main} will then follow as a consequence of the resulting renewal structure.

Hereafter, we fix $v_\star \in (0,v_\bullet)$ and take $k_\star \in \N$ and $\epsilon_\star \in (0,1)$ as in Proposition~\ref{prop:quencheddev}.
We then define  $\rho_\star := \rho(k_\star, \epsilon_\star)$ as given by Proposition~\ref{prop:enoughparticles} and Theorem~\ref{thm:ballisticity}.
We will also fix $\rho \geq \rho_\star$ and write $\P := \P^\rho$ from now on.
By Assumption (R) and Proposition~\ref{prop:quencheddev}, we may assume that
\begin{equation}\label{e:defpstar}
p_\star := p_{\bullet}(k_\star) =  \inf_{k \ge k_\star} \alpha(k, x_\bullet) > 0,
\end{equation}
see also \eqref{e:defpbulletk}.

Let
\begin{equation}\label{e:defbarv}
\widehat{v}_\star := v_\star \wedge \tfrac12 \quad \text{ and } \quad \bar{v} := \tfrac13 \widehat{v}_\star. 
\end{equation}
For $y \in \R^d \times \R$, we define the following space-time regions:
\begin{align}
\um(y)  & = y+\um(0,0), \;\, \;\, \um(0,0) := \{(x,n) \in \mathbb{Z}^d \times \Z_+ \colon x \cdot e_1 \geq \bar v n, |x| \le \mathfrak{R} n\},\\
\tres(y) & = y+\tres(0,0), \; \; \tres(0,0) := \{(x,n) \in \mathbb{Z}^d \times \Z_- \colon\, x \cdot e_1 < \bar v n\},
\end{align}
where $\mathfrak{R}$ is as in Assumption~(S).
As in \cite{HHSST14}, we define the sets of trajectories
\begin{equation}
\begin{aligned}
W_{y}^\um &= \text{ trajectories in } W \text{ that intersect $\um(y)$ but not $\tres(y)$},\\
W_{y}^\tres &= \text{ trajectories in } W \text{ that intersect $\tres(y)$ but not $\um(y)$},\\
W_{y}^\treze &= \text{ trajectories in } W \text{ that intersect both $\um(y)$ and $\tres(y)$}.
\end{aligned}
\end{equation}
Note that $W^\um_y$, $W^\tres_y$ and $W^\treze_y$ are disjoint,
and therefore the sigma-algebras
\begin{equation}
\label{e:sigmaalgebrastraj}
\mathcal{G}^{I}_{y} := \sigma \left( \omega(A) \colon\, A \subset W^{I}_{y}, A \in \cW  \right),
I = \um,\tres,\treze,
\end{equation}
are jointly independent under $\P$. Define also the sigma-algebras
\begin{equation}
\label{e:sigmaalgebraunif}
\begin{aligned}
\mathcal{U}^{\um}_{y} & = \sigma \left( U_z \colon\, z \in \um(y) \right),\\
\mathcal{U}^{\tres}_{y} & = \sigma \left( U_{z} \colon\, z \in \tres(y) \right),
\end{aligned}
\end{equation}
and set
\begin{equation}
\label{e:sigmaalgebraFxt}
\mathcal{F}_{y} = \sigma \Big( \mathcal{G}^{\tres}_{y}, \mathcal{G}^{\treze}_{y}, \mathcal{U}^{\tres}_{y} \Big).
\end{equation}
Note that, for two space-time points $y, y' \in \Z^d \times \Z$, 
if $y \in \tres(y')$ then $\cF_{y} \subset \cF_{y'}$.

In order to define the regeneration time, we first need to introduce certain \emph{record times} $(R_k)_{k\in \N}$.
The definition here will be different from the one in \cite{HHSST14}.
To this end, set $R_0 := 0$ and, recursively for $k \in \N_0$,
\begin{align}
\label{e:records}
R_{k+1} & := \inf \left\{n \ge R_k+1 \colon\, \left( X_n -X_{R_k} \right)\cdot e_1 > \bar{v} (n-R_k)  \right\}.
\end{align}
Note that $(X_{R_k+1} - X_{R_k}) \cdot e_1 > 0$ if and only if $R_{k+1} = R_k+1$.

Define now a filtration $(\mathcal{F}_k)_{k \geq 0}$
by setting, for $k \ge 0$,
\begin{equation}
\label{e:filtration2}
\begin{aligned}
\mathcal{F}_k := 
& \Big\{B \in \sigma(\omega,U) \colon\, \\
& \quad \forall \, y \in \Z^d \times \Z, \, \exists \, B_{y} \in \mathcal{F}_{y}
\text{ with } B \cap \{Y_{R_k} = y\} = B_{y} \cap \{Y_{R_k} = y\} \Big\},
\end{aligned}
\end{equation}
i.e., the sigma-algebra generated by $Y_{R_k}$, all $U_z$ with $z \in \tres(Y_{R_k})$
and all $\omega(A)$ such that $A \subset W^\tres_{Y_{R_k}} \cup W^\treze_{Y_{R_k}}$.
In particular, $(Y_i)_{0 \le i \le R_k} \in \mathcal{F}_k$.

Finally we define the event
\begin{equation}
\label{e:Axt}
A^{y} = \big\{Y^{y}_i \in \um(y) \,\,\forall\,i \in {\mathbb{Z}_+} \big\}
\end{equation}
in which the walker started at $y$ remains inside $\um(y)$,
the probability measure
\begin{equation}
\P^{\um} (\cdot) = \P \left(~\cdot~ {\big |}
~\omega\big(W^{\treze}_{0}\big)=0,\, A^{0}\right)
\label{e:pmarrom}
\end{equation}
with corresponding expectation operator $\mathbb{E}^{\um}$, the \emph{regeneration record index}
\begin{equation}
\label{e:regrec}
\mathcal{I} = \inf\Big\{ k \in \N \colon\, \omega\big(W^{\treze}_{Y_{R_k}}\big)= 0,\,
A^{Y_{R_k}} \text{ occurs } \Big\}
\end{equation}
and the \emph{regeneration time}
\begin{equation}
\label{e:regtime}
\tau = R_{\mathcal{I}}.
\end{equation}
The following two theorems are the analogous of Theorems~4.1--4.2 of \cite{HHSST14} in our setting.
\begin{theorem}
\label{t:regeneration}
Almost surely on the event $\{\tau < \infty\}$, the process $(Y_{\tau+i} - Y_\tau)_{i \in\Z_+}$
under either the law $\P (~\cdot \mid \tau, (Y_i)_{0\le i \le \tau})$ or $\P^{\um}
(~\cdot \mid \tau, (Y_i)_{0 \le i \le \tau})$ has the same distribution as that of $(Y_i)_{i \in\Z_+}$
under $\P^\um(\cdot)$.
\end{theorem}

\begin{theorem}
\label{t:tailregeneration}
There exists a constant $c > 0$ such that
\begin{equation}
\label{e:tailregeneration}
\E \left[e^{c (\log \tau)^{3/2}} \right] < \infty
\end{equation}
and the same holds with $\mathbb{E}^{\um}$ replacing $\mathbb{E}$.
\end{theorem}
The proof of Theorem~\ref{t:regeneration} follows exactly as that of Theorem~4.1 in \cite{HHSST14} and thus we omit it here.
Theorem~\ref{t:tailregeneration} will be proved in the next section.
From them follows the:

\begin{proof}[Proof of Theorem~\ref{thm:main}]
Using Theorems~\ref{t:regeneration}--\ref{t:tailregeneration},
one may follow almost word for word the arguments given in Section~4.3 of \cite{HHSST14},
with the difference of having now random vectors instead of real-valued random variables.
In particular, we obtain the formulas
\begin{align}
v & = \frac{\E^\um \left[X_\tau \right]}{\E^\um[\tau]},\label{e:formulav}\\ 
\Sigma_{i,j} & = \frac{\E^\um \left[\left(X_\tau - v \tau \right) \cdot e_i \left(X_\tau - v \tau \right) \cdot e_j \right]}{\E^\um \left[ \tau \right]} \label{e:formulaSigma}
\end{align}
for the velocity $v$ and the covariance matrix $\Sigma$, from which the comments made after Theorem~\ref{thm:main} may be deduced.
The fact that $v \cdot e_1 \ge v_\star$ follows from Theorem~\ref{thm:ballisticity}.
\end{proof}

\subsection{Control of the regeneration time}
\label{ss:tailregeneration}

In this section, we prove Theorem~\ref{t:tailregeneration} by adapting Section~4.2 of \cite{HHSST14} to our setting.
The two most important modifications are as follows.
First, in order to bypass the requirement of uniform ellipticity,
we do not require the random walker to make jumps in a fixed direction independently of the environment
but instead only over points containing enough particles.
For this, we need to estimate the probability of certain joint occupation events, cf.\ Lemma~\ref{l:probfill} below.
Second, we need a substitute for Lemma~4.5 of \cite{HHSST14},
which gave a quenched estimate on the backtrack probability of the random walker and was obtained therein using a monotonicity
property only available in one dimension. This is the role of Lemma~\ref{l:no_top} below, obtained with the help of Propositions~\ref{prop:enoughparticles}--\ref{prop:quencheddev}.
 
In our first lemma, we construct a path for the random walk to follow where all the points have a large number of particles.
This has a cost that is at most exponential.
\begin{lemma}\label{l:probfill}
\newconstant{c:probfill}
There exists $\useconstant{c:probfill} \in (0,1)$ such that,
for all $L \in \N$,
\begin{equation}\label{e:probfill}
\P \left( \omega\left(W_{i x_\bullet, i} \setminus (W^\treze_{L x_\bullet, L} \cup W^\treze_0) \right) \ge k_\star \;\forall\; i=0, \ldots, L-1 \right) \ge \useconstant{c:probfill}^{L}.
\end{equation}
\end{lemma}
\begin{proof}
We proceed by induction in $L$.
Recall the definition of $S^{0, 1}$ in Section~\ref{s:construction} and let
\begin{align}\label{e:prprobfill2}
\useconstant{c:probfill} : = \P\left( N(0) \ge k_\star \right) \P\left( S^{0, 1}_n \notin \tres(0) \cup \um(x_\bullet, 1) \;\forall\; n \in \Z
\right)^{k_\star} > 0.
\end{align}
Since $\P \left(  \omega \left(W_0 \setminus ( W^\treze_{x_\bullet, 1} \cup W^{\treze}_0) \right) \ge k_\star \right) \ge \useconstant{c:probfill}$, the claim holds for $L=1$.
Assume that it holds for some $L \ge 1$.
Noting that $\um((i+1) x_\bullet, i+1) \subset \um(i x_\bullet, i)$, write
\begin{align}\label{e:prprobfill1}
  \P \Big( \bigcap_{i=0}^{L} & \left\{ \omega\left(W_{i x_\bullet, i} \setminus (W^\treze_{(L+1) x_\bullet, L+1} \cup W^\treze_0) \right) \ge k_\star \right\} \Big) \nonumber\\
  \ge \; & \P \Bigg( \bigcap_{i=0}^{L-1}\left\{\omega\left(W_{i x_\bullet, i} \setminus (W^\treze_{L x_\bullet, L} \cup W^\treze_0) \right) \ge k_\star \right\} \cap \nonumber\\[-15pt]
  & \qquad \quad \left\{ \omega\left(W_{L x_\bullet, L} \setminus (W^\treze_{(L+1) x_\bullet, L+1} \cup W^\treze_{L x_\bullet, L}) \right) \ge k_\star \right\} \Bigg).
\end{align}
Using now that, for any $i=0, \ldots, L-1$, the sets of trajectories $W_{i x_\bullet, i} \setminus (W^\treze_{L x_\bullet, L} \cup W^\treze_0)$ and $W_{L x_\bullet, L} \setminus (W^\treze_{(L+1) x_\bullet, L+1} \cup W^\treze_{L x_\bullet, L})$ are disjoint, 
and using also the translation invariance of $\P$, we see that the right-hand side of \eqref{e:prprobfill1} equals
\begin{equation}
\P \left(\bigcap_{i=0}^{L-1}\left\{\omega\left(W_{i x_\bullet, i} \setminus (W^\treze_{L x_\bullet, L} \cup W^\treze_0) \right) \ge k_\star \right\} \right)
\P \left( \omega\left(W_0 \setminus (W^\treze_{x_\bullet, 1} \cup W^\treze_0 ) \right) \ge k_\star \right) \geq \useconstant{c:probfill}^{L+1}
\end{equation}
by the induction hypothesis, concluding the proof.
\end{proof}

Our next result is an estimate on the conditional backtrack probability of the random walker, 
which as already mentioned can be seen as a substitute for Lemma~4.5 of \cite{HHSST14}. 
Recall the definition of $\widehat{v}_\star = v_\star \wedge \tfrac12$.
\newconstant{c:height}
\begin{lemma}
\label{l:no_top}
There exists a constant $\useconstant{c:height} > 0$ such that
\begin{equation}
\label{e:no_top}
\P \left( Y^y_n - y \notin \cH_{\widehat{v}_\star, 0} \;\forall\; n \in \N \;\middle| \; \mathcal{F}_y \right)
\ge \useconstant{c:height} \quad \P\text{-a.s.} \;\; \forall\,y \in \Z^d \times \Z.
\end{equation}
\end{lemma}
\begin{proof}
For $y \in \Z^d \times \Z$ and $L \in \N$, write $y_{(L)} := y + (L x_\bullet, L)$ and let
\begin{align}\label{e:prnotop3}
\cB^L_y & := \bigcap_{i=0}^{L-1} \left\{ \omega\left(W_{y + (i x_\bullet, i)} \setminus (W^\treze_{y_{(L)}} \cup W^\treze_{y}) \right) \ge k_\star \right\},\nonumber\\
\cC^L_y & := \bigcap_{i=0}^{L-1} \left\{ U_{y + (i x_\bullet, i)} \le p_\star \right\}.
\end{align}
Recall \eqref{e:defpbulletk} and the discussion below it.
Put $L_\star := \lfloor (1-\widehat{v}_\star) L \rfloor$
and abbreviate $\cA^{L_\star}_{y}:= \cA^{L_\star, \widehat{v}_\star, k_\star, \varepsilon_\star}_{y}$ (cf.\ \eqref{e:defcA}).
Note that $\mathcal{B}^L_y$, $\cC^L_y$ are measurable in $\cF_{y_{(L)}}$ to obtain, $\P$-a.s.,
\begin{align}
\label{e:prnotop1}
  \P \Big( Y^y_i - y & \notin \mathcal{H}_{\widehat{v}_\star,0} \,\forall\, i \in \N \;\Big| \; \mathcal{F}_y \Big) \nonumber\\
  \ge \; & \P \left(\cC^L_y, \cB^L_y, (\cA^{L_\star}_{y_{(L)}})^c, \; Y^{y_{(L)}}_n - y_{(L)} \notin \cH_{\widehat{v}_\star, L_\star} \;\forall\, n \in \N \;\middle|\; \mathcal{F}_y \right)\nonumber\\
  = \; & \E \left[ \mathbbm{1}_{\cC^L_y} \mathbbm{1}_{\cB^L_y} \, \P \left( (\cA^{L_\star}_{y_{(L)}})^c, \; Y^{y_{(L)}}_n - y_{(L)} \notin \cH_{\widehat{v}_\star, L_\star} \;\forall\, n \in \N \; \,\middle| \cF_{y_{(L)}} \right) \Big| \; \mathcal{F}_y \right].
\end{align}
Now, since $\cA^{L_\star}_{y_{(L)}}$, $Y^{y_{(L)}}$ are independent of $\cU^{\tres}_{y_{(L)}}$,
\begin{align}\label{e:prnotop1.5}
& \P \left( (\cA^{L_\star}_{y_{(L)}})^c, \; Y^{y_{(L)}}_n - y_{(L)} \notin \cH_{\widehat{v}_\star, L_\star} \;\forall\, n \in \N \; \,\middle| \cF_{y_{(L)}} \right) \nonumber\\
= \, & \P \left( (\cA^{L_\star}_{y_{(L)}})^c, \; Y^{y_{(L)}}_n - y_{(L)} \notin \cH_{\widehat{v}_\star, L_\star} \;\forall\, n \in \N \; \,\middle| \cG^\um_{y_{(L)}} \vee \cG^\treze_{y_{(L)}} \right) \nonumber\\
\ge \, & (1-c^{-1} e^{-c L_\star}) \P \left( (\cA^{L_\star}_{y_{(L)}})^c \,\middle| \cF_{y_{(L)}} \right) \;\;\text{ a.s.}
\end{align}
by Proposition~\ref{prop:quencheddev}  (recall that $\widehat{v}_\star \le v_\star$).
Substituting this back into \eqref{e:prnotop1} and using that $\mathcal{B}_y^{L}$, $\cA^{L_\star}_{y_{(L)}} \in \sigma(\omega)$,
$\cC^L_y \in \sigma(U)$, we obtain that \eqref{e:prnotop1} is a.s.\ larger than
\begin{align}\label{e:prnotop2}
%& p_\star^L (1-c^{-1} e^{-c L_\star}) \,  \P \left( \cB^L_y , (\cA^{L_\star}_{y_{(L)}})^c \,\middle|\, \cG^\tres_y \vee \cG^\treze_y\right) %\nonumber\\
\tfrac12 p_\star^L \, \P \left(\cB^L_y , (\cA^{L_\star}_{y_{(L)}})^c \,\middle|\, \cG^\tres_y \vee \cG^\treze_y \right)
\end{align}
when $L$ is large enough.
%Note that $\cB^L_y$ and $(\cA^{L_\star}_{y_{(L)}})^c$ are both measurable in $\cG^\um_{y_{(L)}} \vee \cG^\treze_{y_{(L)}}$.
Reasoning as for equation (4.16) in \cite{HHSST14}, we see that, $\P$-a.s.,
\begin{equation*}
  1_{\{\omega(W^{\treze}_{y}) = 0\}} \P \left( \cB^L_y, (\cA^{L_\star}_{y_{(L)}})^c \;\middle| \; \cG^\tres_y \vee \cG^\treze_y \right) = 1_{\{\omega(W^{\treze}_{y}) = 0\}} \P \left( \cB^L_0, (\cA^{L_\star}_{L x_\bullet, L})^c \;\middle| \; \omega(W^{\treze}_0) =0 \right).
\end{equation*}
Moreover, since $\cB^L_y \cap (\cA^{L_\star}_{y_{(L)}})^c$ is non-decreasing (in the sense of Definition~\ref{d:monotone_FKG}),
its conditional probability given $\cG^\tres_y \vee \cG^\treze_y$ only increases if $\omega(W^\treze_y) \neq 0$.
Hence, $\P$-a.s.,
\begin{align}\label{e:prnotop5}
\P \left( \cB^L_y, (\cA^{L_\star}_{y_{(L)}})^c \;\middle| \; \cG^\tres_y \vee \cG^\treze_y \right)
& \ge  \P \left( \cB^L_0 ,(\cA^{L_\star}_{L x_\bullet, L})^c \;\middle|\; \omega(W^{\treze}_0) =0 \right) \nonumber\\
& = \P \left( \cB^L_0, (\widehat{\cA}^{L_\star}_{L x_\bullet, L})^c \right)
\end{align}
where
\begin{equation}\label{e:prnotop6}
\widehat{\cA}^{L_\star}_{L x_\bullet, L} := \Big\{ \exists \, \ell \ge L_\star/(2 \mathfrak{R}), \, \sigma \in \mathfrak{S}^{\widehat{v}_\star, L_\star}
\colon\, \sum_{i=0}^{\ell-1} \mathbbm{1}_{\{\omega(W_{\sigma(i)+L x_\bullet, i+L} \setminus W_0) \ge k_\star \}} < (1-\varepsilon_\star) \ell \Big\}.
\end{equation}
Since $\cB^L_0$ and $(\widehat{\cA}^{L_\star}_{L x_\bullet, L})^c$ are functions of $\omega$ only and are both non-decreasing, it follows from Proposition~\ref{prop:FKG} and Lemma~\ref{l:probfill} that \eqref{e:prnotop5} is at least
\begin{equation}\label{e:prnotop7}
\P \left( \cB^L_0 \right) \P\left( (\widehat{\cA}^{L_\star}_{L x_\bullet, L})^c \right)
\ge \useconstant{c:probfill}^L \P\left( (\cA^{L_\star}_{L x_\bullet, L})^c \;\middle|\; \omega(W^{\treze}_0) =0  \right).
\end{equation}
Now note that, by Proposition~\ref{prop:enoughparticles},
\begin{equation}\label{e:prnotop8}
\P\left( \cA^{L_\star}_{L x_\bullet, L} \;\middle|\; \omega(W^{\treze}_0) =0  \right)
\le \frac{\P\left( \cA^{L_\star}_{L x_\bullet, L} \right)}{\P\left(\omega(W^{\treze}_0) =0  \right)}
\le c^{-1} e^{- c(\log L)^{3/2}}
\end{equation}
for some constant $c > 0$.
For fixed $L$  large enough, \eqref{e:prnotop8} is smaller than $1/2$,
and thus \eqref{e:no_top} follows from \eqref{e:prnotop1}--\eqref{e:prnotop8} with $\useconstant{c:height} := \tfrac14 (\useconstant{c:probfill} p_\star)^L$.
\end{proof}

We proceed with the adaptation of Section~4.2 of \cite{HHSST14}.
As in equation (4.21) therein, we define the \emph{influence field}
\begin{equation}\label{e:definfluencefield}
h(y) := \inf \left\{l \in \Z_+ \colon\, \omega(W^\treze_y \cap W^\treze_{y+(l x_\bullet, l)}) = 0  \right\}, \;\;\;\; y \in \Z^d \times \Z.
\end{equation}
Using $x_\bullet \cdot e_1 \ge 1$ and similar arguments as for Lemma~4.3 in \cite{HHSST14}, we obtain:
\begin{lemma}\label{l:hxt_exp}
\newconstant{c:h_xt1}
\newconstant{c:h_xt2}
There exist constants $\useconstant{c:h_xt1}, \useconstant{c:h_xt2} > 0$
such that, for all $y \in \Z^d \times \Z$,
\begin{equation}
\label{e:h_xt_exp}
\P \left( h(y) > l \right) \leq \useconstant{c:h_xt1} e^{-\useconstant{c:h_xt2} l}, \qquad l\in{\mathbb{Z}_+}.
\end{equation}
\end{lemma}

Let
\begin{equation}\label{e:defhatp}
\hat{p} := \useconstant{c:probfill} p_\star > 0
\end{equation}
where $p_\star$ is as in \eqref{e:defpstar} and $\useconstant{c:probfill}$ as in \eqref{e:probfill}.
Analogously to equations (4.28)--(4.29) in \cite{HHSST14}, we set, for $T>1$,
\begin{equation}\label{e:def.deltaepsilonT'T''}
\delta:= \left( - 4 \log( \hat{p} ) \right)^{-1}, \quad \epsilon := \frac{1}{4(d+1)}(\useconstant{c:h_xt2} \delta \wedge 1),
\quad T' = \lfloor T^\epsilon \rfloor,\quad T'' = \lfloor \delta \log(T) \rfloor,
\end{equation}
and we define the \emph{local influence field} at a space-time point $y \in \Z^d \times \Z$ to be:
\begin{equation}\label{e:deflocalfield}
h^T(y):= \inf \left\{l \in \Z_+ \colon\, \omega \big( W^\treze_y \cap W^\treze_{y + (l x_\bullet, l )} \cap W^\um_{y-(\lfloor (\bar{v}/\mathfrak{R} )T' \rfloor x_\bullet, \lfloor (\bar{v}/\mathfrak{R} )T' \rfloor )} \big) = 0 \right\}.
\end{equation}

Note that our definition is slightly different from that of \cite{HHSST14}.
As in Lemma~4.4 therein, we obtain:
\begin{lemma}
\label{l:locinfl}
For all $T > 1$ and all $y \in \Z^d \times \Z$,
\begin{equation}
\label{e:locinfl}
\P \left( h^T(y) > l \;\middle|\; \cF_{y-(\lfloor (\bar{v}/\mathfrak{R} )T' \rfloor x_\bullet, \lfloor (\bar{v}/\mathfrak{R} )T' \rfloor )} \right)
\le \useconstant{c:h_xt1} e^{-\useconstant{c:h_xt2} l} \;\;\; \forall \,l \in \Z_+ \;\;\; \P\text{-a.s.,}
\end{equation}
where $\useconstant{c:h_xt1}, \useconstant{c:h_xt2}$ are the same constants as in
Lemma~{\rm \ref{l:hxt_exp}}.
\end{lemma}

\begin{proof}
Note that $h^T(y)$ is independent of $\cF_{y-(\lfloor (\bar{v}/\mathfrak{R} )T' \rfloor x_\bullet, \lfloor (\bar{v}/\mathfrak{R} )T' \rfloor)}$ and $h^T(y) \le h(y)$.
\end{proof}

As in \cite{HHSST14}, an important definition is that of a \emph{good record time} (g.r.t.):
for $k \in \N$, we call $R_k$ a g.r.t.\ if
\begin{align}
\label{e:good_record1}
& \,\; h^T(Y_{R_k}) \leq T'',\\
\label{e:good_record2}
&
\begin{array}{rcl}
\omega\left(W_{Y_{R_k}+(l x_\bullet, l)} \setminus (W^\treze_{Y_{R_k}} \cup W^\treze_{Y_{R_k}+(T'' x_\bullet, T'' )}) \right) & \ge & k_\star \\
\text{ and } \qquad U_{Y_{R_k} + (lx_\bullet,l)} & \leq & p_\star \end{array} \quad \forall\,l = 0, \dots, T''-1,\\
\label{e:good_record3}
& \,\; \omega(W^{\treze}_{Y_{R_k}+(T''x_\bullet, T'')} \cap W^\um_{Y_{R_k}} ) = 0,\\
\label{e:good_record4}
& \,\; Y_n \in \um(Y_{R_{k+T''}}) \text{ for all } n \in \{R_{k+T''}, \ldots, R_{k+T'}\}.
\end{align}
Note that, when \eqref{e:good_record2} occurs, $Y_{R_k}+(T''x_\bullet,T'') = Y_{R_{k+T''}}$.

With the above definitions and results in place,
only minor modifications are required to adapt the rest of Section~4.2 of \cite{HHSST14} to our setting.
For completeness, we provide below all the details.

The following proposition is the main step in the proof of Theorem~\ref{t:tailregeneration}.
%%%%%
\newconstant{c:manygrts}
%%%%%
%
\begin{proposition}
\label{p:manygrts}
There exists a constant $\useconstant{c:manygrts} >0$ such that, for all $T>1$ large enough,
\begin{equation}
\mathbb{P}\left[\text{$R_k$ is not a g.r.t.\ for all $1\le k \le T$ }\right]
\leq e^{ -\useconstant{c:manygrts} T^{1/2}}.
\end{equation}
\end{proposition}
\begin{proof}
First we claim that there exists a $c>0$ such that, for any $k > T'$,
\begin{equation}
\label{e:saw_pemba}
\mathbb{P} \left[ R_k \text{ is a g.r.t.} \big| \mathcal{F}_{k-T'} \right]
\geq c T^{\delta \log( \hat{p})}  \text{ a.s.}
\end{equation}
To prove \eqref{e:saw_pemba}, we will find $c>0$ such that
\begin{align}
\label{e:good_cond1}
& \mathbb{P} \big[ \text{\eqref{e:good_record1}} \; \big| \; \mathcal{F}_{k-T'} \big] \geq c
 & \text{ a.s.,}\\
\label{e:good_cond2}
& \mathbb{P} \big[ \text{\eqref{e:good_record2}} \; \big| \; \mathcal{F}_k \big]
\geq T^{\delta \log(\hat{p})}  & \text{ a.s.,}\\
\label{e:good_cond3}
& \mathbb{P} \big[ \text{\eqref{e:good_record3}} \; \big| \; \text{\eqref{e:good_record2}},
\mathcal{F}_k \big] \geq c  & \text{ a.s.,}\\
\label{e:good_cond4}
& \mathbb{P} \big[ \text{\eqref{e:good_record4}} \; \big| \; \mathcal{F}_{k+T''} \big] \geq c
 & \text{ a.s.\ }
\end{align}

Proof of \eqref{e:good_cond1}:
For $B \in \mathcal{F}_{k-T'}$, write
\begin{align} \label{e:manygrts1}
& \P \left(h^T(Y_{R_k}) > T'', B \right)
= \sum_{y_1, y_2 \in \Z^d \times \Z} \P \left(h^T(y_2) > T'', Y_{R_k}
= y_2, Y_{R_{k-T'}}=y_1, B_{y_1} \right).
\end{align}
Note that, if $Y^{Y_{R_{k-T'}}}_n -Y_{R_{k-T'}} \notin \cH_{\widehat{v}_\star, 0}$ for all $n \in \N$,
then $R_{k} \le R_{k-T'} + C T'$ for some constant $C \ge \mathfrak{R} \ge 1$, and 
moreover $Y_{R_{k-T'}} \in \tres\left( Y_{R_k}- (\lfloor (\bar{v}/\mathfrak{R} )T' \rfloor x_\bullet, \lfloor (\bar{v}/\mathfrak{R} )T' \rfloor) \right)$.
Thus we may upper-bound \eqref{e:manygrts1} by
\begin{align}\label{e:manygrts2}
& \sum_{y_1 \in \Z^d \times \Z} \;\,
\sum_{\substack{y_2 \in \Z^d \times \Z \colon |y_2 - y_1|_\infty \le C T', \\ y_1 \in \tres\left(y_2 - (\lfloor (\bar{v}/\mathfrak{R} )T' \rfloor x_\bullet, \lfloor (\bar{v}/\mathfrak{R} )T' \rfloor) \right)}}
\P \left(h^T(y_2) > T'', Y_{R_{k-T'}}=y_1, B_{y_1} \right) \nonumber\\
& \quad + \sum_{y_1 \in \Z^d \times \Z} \P \left( \exists\, n \in \N \colon\, Y^{y_1}_n - y_1 \in \cH_{\widehat{v}_\star,0}, \, Y_{R_{k-T'}}=y_1, B_{y_1} \right) \nonumber\\
\le \; & \left\{ \hat{C} (T')^{d+1} \useconstant{c:h_xt1} e^{-\useconstant{c:h_xt2} T''}
+ 1-\useconstant{c:height} \right\}  \P \left( B \right)
\le \left\{  \hat{C} \useconstant{c:h_xt1} e^{\useconstant{c:h_xt2}} T^{-\frac34 \delta
\useconstant{c:h_xt2}} + 1-\useconstant{c:height} \right\}  \P \left( B \right)
\end{align}
for some constant $\hat{C}> 0$,
where for the first inequality we use Lemmas~\ref{l:no_top} and~\ref{l:locinfl} 
(see also the comment after \eqref{e:sigmaalgebraFxt}), 
and for the second we use the definition of $\epsilon$.
Thus, for $T$ large enough, \eqref{e:good_cond1} is satisfied with $c = \useconstant{c:height}/2$.

Proof of \eqref{e:good_cond2}:
Let $\cB^L_y$ as in \eqref{e:prnotop3} and note that
$\cB^{T''}_y$, $(U_{y+(l x_\bullet, l)})_{l \in \Z_+}$ and $\cF_y$ are jointly independent.
For $B \in \cF_k$, write
\begin{align}\label{e:manygrts3}
\P \left( \cB^{T''}_{Y_{R_k}}, U_{Y_{R_k} + (lx_\bullet,l)} \leq p_\star, \, B\right)
& = \sum_{y \in Z^d \times \Z} \P \left( \cB^{T''}_y, U_{y + (lx_\bullet,l)} \leq p_\star, \, Y_{R_k}=y, B_y \right) \nonumber\\
& = p_\star^{T''} \P \left(\cB^{T''}_0\right) \P \left( B \right)
\end{align}
to conclude that \eqref{e:good_record2} is independent of $\cF_k$.
Then \eqref{e:good_cond2} follows by Lemma~\ref{l:probfill} and \eqref{e:def.deltaepsilonT'T''}.

Proof of \eqref{e:good_cond3}: We may ignore the conditioning on \eqref{e:good_record2} since
this event is independent of \eqref{e:good_record3} and $\cF_k$.
For $B \in \cF_k$, write
\begin{align}
\label{e:manygrts4}
& \mathbb{P} \left(\omega(W^{\um}_{Y_{R_k}}
\cap W^{\treze}_{Y_{R_k}+(T'', T'')} ) =0, B\right) = \sum_{y \in \Z^2} \mathbb{P} \left(\omega(W^{\um}_{y}
\cap W^{\treze}_{y+(T'', T'')} ) =0, Y_{R_k} = y, B_y\right) \nonumber\\
& = \sum_{y \in \Z^2} \mathbb{P} \left(\omega(W^{\um}_{y}
\cap W^{\treze}_{y+(T'', T'')} ) =0 \right) \mathbb{P}\left( Y_{R_k} = y, B_y\right)
\ge \mathbb{P}\left(\omega(W^{\treze}_0) = 0 \right) \mathbb{P} \left( B \right),
\end{align}
where the second equality uses the independence between $\cG^{\um}_y$ and $\cF_y$.

Proof of  \eqref{e:good_cond4}: For $B \in \mathcal{F}_{k+T''}$, write
\begin{align}
\label{e:manygrts5}
& \P \left( Y_n \in \um(Y_{R_{k+T''}}) \;\forall\; R_{k+T''} \le n \le R_{k+T'}, B \right) \nonumber\\
\ge \; & \sum_{y \in \Z^2} \P \left( Y^{y}_n \notin \cH_{\widehat{v}_\star, 0} \;\forall\; n \in \N, Y_{R_{k+T''}}=y, B_y \right) \nonumber\\
\ge \; & \useconstant{c:height} \P \left( B \right)
\end{align}
by Lemma~\ref{l:no_top}.

Thus, \eqref{e:saw_pemba} is verified. Since $\{R_k \text{ is a g.r.t.}\} \in \mathcal{F}_{k+T'}$,
we obtain, for $T$ large enough,
\begin{align}
\label{e:manygrts6}
\mathbb{P} \left( R_k \text{ is not a g.r.t.\ for any } k \le T \right)
& \le \mathbb{P} \left( R_{(2k+1)T'} \text{ is not a g.r.t.\ for any } k \le T/3T' \right) \nonumber\\
& \le \exp \left\{ -\frac{c}{4} \frac{T^{1+{\delta \log(\hat{p})}}}{T'} \right\} \nonumber\\
& \le \exp \left\{-\frac{c}{4}T^{\frac12} \right\}
\end{align}
by our choice of $\epsilon$ and $\delta$.
\end{proof}

The proof of Theorem~\ref{t:tailregeneration} can then be finished as in \cite{HHSST14}.

\begin{proof}[Proof of Theorem~{\rm \ref{t:tailregeneration}}]
Since $\P^{\um}(\cdot) = \P(\cdot | A^0, \omega(W^\treze_0) = 0)$ and
$\P(A^0, \omega(W^\treze_0) = 0)>0$, it is enough to prove the statement under
$\P$. To that end, let
\begin{equation}
\label{prtailreg1}
\begin{aligned}
E_1 & = \{\exists \; y \in [-2 \mathfrak{R} T, 2 \mathfrak{R} T]^d \times [-T,T] \cap \Z^d \times \Z \colon\, h(y) \ge \lfloor (\bar{v}/\mathfrak{R}) T' \rfloor \},\\
E_2 & = \{\exists \; y \in [-2 \mathfrak{R} T, 2\mathfrak{R} T]^d \times [-T,T] \cap \Z^d \times \Z \colon\, Y^y \text{ touches } y
+ \mathcal{H}_{\widehat{v}_\star, \lfloor \bar{v} T' \rfloor}\}.
\end{aligned}
\end{equation}
Then, by Lemma~\ref{l:hxt_exp}, \eqref{e:LD} and a union bound, there exists
a $c>0$ such that
\begin{equation}\label{prtailreg2}
\P \left( E_1 \cup E_2 \right) \le c^{-1} e^{-c  (\log T)^{3/2}} \qquad \forall\,T>1.
\end{equation}
Next we argue that, for all $T$ large enough, if $R_k$ is a good record time with $k \le (\bar{v}/\mathfrak{R}) T$ and
both $E_1$ and $E_2$ do not occur then $\tau \le R_{k+T''} \le T$. Indeed, if $T'' 
\le \bar{v} T/\mathfrak{R}$, then on $E_2^c$ we have $R_{\lfloor (\bar{v}/\mathfrak{R}) T \rfloor +T''} \le T$ since
otherwise $Y$ touches $\mathcal{H}_{\widehat{v}_\star, \lfloor \bar{v} T' \rfloor}$.
Thus we only need to verify that
\begin{equation}
\label{prtailreg3}
\omega(W^\treze_{Y_{R_{k+T''}}})=0
\end{equation}
and that
\begin{equation}
\label{prtailreg4}
A^{Y_{R_{k+T''}}} \text{ occurs }
\end{equation}
under the conditions stated.

To verify \eqref{prtailreg4}, note that, on $E_2^c$, 
we have $Y_{R_{k+T''}} \in [-2 \mathfrak{R} T, 2 \mathfrak{R} T]^d \times[0,T] \cap \Z^d \times \Z$ 
and, moreover, if $T'' < \tfrac12 T'$ then
\begin{equation}\label{prtailreg4.5}
Y_{R_{k+T'}+l} \in \um(Y_{R_{k+T''}}) \; \forall \; l \in \Z_+,
\end{equation}
which together with \eqref{e:good_record4} implies \eqref{prtailreg4}.

To verify \eqref{prtailreg3}, first note that, by \eqref{e:good_record1} and \eqref{e:good_record3},
it is enough to check that
\begin{equation}
\label{prtailreg5}
\omega(W^\treze_{Y_{R_k}} \cap W^{\treze}_{Y_{R_k} - (\lfloor (\bar{v}/\mathfrak{R}) T' \rfloor x_\bullet, \lfloor (\bar{v}/\mathfrak{R}) T' \rfloor)} ) = 0
\end{equation}
on $E_1^c \cap E_2^c$.
Noting that, on $E_2^c$, $Y_{R_k} - (\lfloor (\bar{v}/\mathfrak{R})T' \rfloor x_\bullet, \lfloor (\bar{v}/\mathfrak{R})T' \rfloor) \in [-2 \mathfrak{R} T, 2 \mathfrak{R} T]^d \times [0,T] \cap \Z^d \times \Z$,
 \eqref{prtailreg5} follows from the definitions of $E_1$ and of $h$.

In conclusion, for $T$ large enough we have
\begin{align}
\label{prtailreg8}
\P \left( \tau > T \right)
& \le \P(E_1 \cup E_2) + \P \left( R_k \text{ is not a g.r.t.\ }
\forall \; k \le {\bar v}T  \right) \nonumber\\
& \le c^{-1}e^{-c(\log T)^{3/2}} + e^{-\useconstant{c:manygrts} ({\bar v} T)^{1/2}}
\end{align}
from which \eqref{e:tailregeneration} follows.
\end{proof}

%%%%%% End of text for regmany%%%%%

\appendix
%%%%%%\subfile{appendixmany}%%%%%
%%%%%% Subile replaced by text below %%%%%
\section{Decoupling of space-time boxes}
\label{s:decouple}

The aim of this section is to prove Theorem~\ref{t:decouple}.
The proof is very similar to the proof of Theorem~C.1 in \cite{HHSST14};
only the most important changes are described here.
In the following subsections, we will concentrate on
intermediate results required for item (b) of Theorem~\ref{t:decouple},
i.e., the case where $f_1$, $f_2$ are both non-decreasing.
The non-increasing case will be discussed in the proof of Theorem~\ref{t:decouple}
itself at the end of this Appendix.

The constants in this section will be all \emph{independent} of $\rho$;
this is crucial for the perturbative arguments of Section~\ref{s:renormalization}.

\subsection{Soft local times}
\label{ss:SLT}

We start with a coupling result.
For a Polish space $\Sigma$ and a Radon measure $\mu$ on $\Sigma$,
let $m$ denote the Poisson point process on $\Sigma \times \R_+$ with
intensity measure $\mu \otimes \text{d} v$, where $\text{d} v$ is the Lebesgue measure on $\R_+$. 
We write $m = \sum_{i \in \N} \delta_{z_i,v_i}$ with $(z_i, v_i) \in \Sigma \times \R_+$.

Fix a sequence of independent $\Sigma$-valued random elements $Z_j$, $j \in \N$.
Assume that the law of $Z_j$ is absolutely continuous with respect to $\mu$ with density $g_j$.
As in Appendix~A of \cite{HHSST14}, we define the \emph{soft local times} $G_j:\Sigma \to [0,\infty)$, $j \in \N$
by setting
\begin{equation}
\label{e:defSLT}
\begin{split}
& \xi_{1} = \inf \big\{ t \geq 0\colon\, tg_1(z_i) \geq v_i \text{ for at least one } i \in \N\big\},\\
& \quad G_{1}(z) = \xi_{1} \, g_1(z),\\
& \qquad \quad \vdots\\
& \xi_{k} =  \inf \big\{ t \geq 0\colon\, tg_j(z_i) + G_{k-1}(z_i) \geq v_i \text{ for at least } k
\text{ indices } i \in \N\big\},\\
& \quad G_{k}(z) = \xi_{1} \, g_1(z) + \dots + \xi_{k} \, g_k(z)
\end{split}
\end{equation}

This construction can be used to prove the following.
\begin{lemma}
\label{l:couplesystem}
The random variables $(\xi_j)_{j \in \N}$ in \eqref{e:defSLT} are i.i.d.\ Exp$(1)$.
Furthermore, 
there exists a coupling $\mathbb{Q}$ of $(Z_j)_{j\in \N}$ and $m$ such that,
for any $J \in \N$, $\rho > 0$,
\begin{equation}
\label{Qineq}
\mathbb{Q} \left[ \mathbbm{1}_{H'}\sum_{j \leq J} \delta_{Z_j} \leq  \mathbbm{1}_{H'} \sum_{i\colon v_i < \rho} \delta_{z_i} \right]
\geq \mathbb{Q} \big[ \sup_{z \in H'} G_{J}(z) \leq \rho \big]
\end{equation}
for all compact $H' \subset \Sigma$.
\end{lemma}
\begin{proof}
Follows from Proposition~A.2 in \cite{HHSST14} (compare to Corollary~A.3 therein).
\end{proof}

\subsection{Simple random walks}
\label{ss:SRW}
As in \cite{HHSST14}, we will need some basic facts about the heat kernel of random walks on $\Z^d$.
Let $p_n(x,x') = P_x(S^{x,1} = x')$, $x,x'\in\Z^d$, with $P_z$, $S^{z,i}$ as defined in Section~\ref{s:construction}. 
Hereafter we will assume that $S^{1,0}$ is lazy; non-lazy $S^{1,0}$ are bipartite,
and the argument below may adapted as outlined in Remark~C.4 of \cite{HHSST14}.
Lazy $S^{1,0}$ are aperiodic in the sense of \cite{LL10}, 
and thus there exist constants $C, c>0$ such that the following hold for all $n \in \N$:
\begin{align}
\label{e:localclt}
& \sup_{x \in \Z^d}  p_n(0,x) \leq \frac{C}{n^{d/2}}, \\
\label{e:SRW1}
& |p_n(0,x) - p_n(0,x')| \le \frac{C |x-x'|}{n^{(d+1)/2}} \;\;\; \forall \; x, x' \in \Z^d, \\
\label{e:SRW2}
& P_0(|S_n| > \sqrt{n} \log n ) \le \, C e^{-c\log^{2} n}.
\end{align}
For \eqref{e:localclt}, see e.g.\ Lawler and Limic~\cite[Theorem 2.4.4]{LL10}.
To get \eqref{e:SRW1}, use \cite[Theorem~2.3.5 and equation (2.2)]{LL10},
 while \eqref{e:SRW2} follows by an application of e.g.\ Azuma's inequality.

The above inequalities will be used to prove Lemma~\ref{l:integration}
below, regarding the integration of the heat kernel over a sparse cloud of sample points.
In order to state it, we need the following definitions.
\begin{definition}
\label{d:balanced}
(a) We say that a collection of intervals $\{C_i\}_{i \in I}$ is an $L$-paving if
\begin{equation}
\label{e:paving}
|C_i| = L^d  \;\; \forall \; i \in I, \quad \;\; \bigcup_{i \in I} C_i = \Z^d \quad \text{ and } \quad C_i \cap C_j = \emptyset \;\; \forall \; i \neq j \in I.
\end{equation}
(b) For $\rho \in (0,\infty)$, we say that a collection of points $(x_j)_{j \in J} \subset \Z^d$ is $\rho$-sparse with
respect to the $L$-paving $\{C_i\}_{i \in I}$ when
\begin{equation}
\label{e:rhodense}
\#\{j \colon x_j \in C_i\} \le \rho L^d \;\;\; \forall \; i \in I.
\end{equation}
\end{definition}
In the above definition, by interval, we mean a subset of $\mathbb{Z}^d$ that is a Cartesian products of intervals of $\mathbb{Z}$.

The next lemma provides an estimate of the sum of the heat kernel over a sparse collection $(x_j)_{j \in J}$.
\begin{lemma}
\label{l:integration}
There exists $c>0$ such that the following holds.
Let $\{C_i\}_{i \in I}$ be an $L$-paving and $(x_j)_{j \in J}$ be $\rho$-sparse collection
with respect to $\{C_i\}_{i \in I}$. Then, for all $n \ge L$,
\begin{equation}
\sum_{j \in J} p_n(0,x_j) \leq \rho \left\{ 1 + \frac{cL (\log n)^d}{\sqrt{n}}\right\}.
\end{equation}
\end{lemma}
\begin{proof}
For each $i \in I$, choose $z_i \in C_i$ such that
\begin{align}
p_n(0,z_i) = \max_{x \in C_i} p_n(0,x).
\end{align}
Then we have
\begin{align}
\label{e:int1}
& \sum_{j \in J} p_n(0,x_j) = \sum_{i \in I} \sum_{j \colon x_j \in C_i} p_n(0,x_j)
\le \sum_{i \in I} \rho L^d p_n(0,z_i) \nonumber\\
& \le \rho \sum_{i \in I} \sum_{x \in C_i} |p_n(0,x) -p_n(0,z_i)| + \rho.
\end{align}
On the other hand, by \eqref{e:SRW1}--\eqref{e:SRW2} we have (since $p_n(0,z_i) \le P_0(S^{0,1} \in C_i)$)
\begin{align}\label{e:int2}
\sum_{i \in I} \sum_{x \in C_i} |p_n(0,x) -p_n(0,z_i)|
& \overset{n \ge L} \le (1+n^d) P_0(|S_n| > \sqrt{n} \log n) + \sum_{|x| \le \sqrt{n} \log n} \frac{cL}{n^{(d+1)/2}} \nonumber\\
& \le c L (\log n)^d / \sqrt{n}
\end{align}
and the claim follows by combining \eqref{e:int1} and \eqref{e:int2}.
\end{proof}

%%%%%%%%%%%%%%%%%%

\subsection{Coupling of trajectories}
\label{ss:couptraj}
Given a sequence of points $(x_j)_{j \in J}$ in $\mathbb{Z}^d$, let $(Z^j_n)_{n \in \Z_+}$, $j \in J$,
be a sequence of independent simple random walks on $\mathbb{Z}^d$ starting at $x_j$, and
let $\bigotimes_{j \in J} P_{x_j}$ denote their joint law. 
The next lemma, analogous to Lemma~B.3 in \cite{HHSST14}, 
provides a coupling of $(Z^j_n)_{j \in J}$ with a product Poisson measure on $\mathbb{Z}^d$. 
\begin{lemma}
\label{l:couple}
There exists a constant $c \ge 1$ such that the following holds.
Let $(x_j)_{j \in J} \subset \mathbb{Z}^d$ be $\rho$-sparse with respect to the $L$-paving
$\{C_i\}_{i \in I}$.
Then for any $\rho' \geq \rho$ there exists a coupling $\mathbb{Q}$ of
$\otimes_{j \in J} P_{x_j}$ and the law of a Poisson point process $\sum_{j' \in J'}
\delta_{Y_{j'}}$ on $\Z^d$ with intensity $\rho'$ such that
\begin{equation}
\mathbb{Q} \left[ \1{H'} \sum_{j \in J} \delta_{Z^j_n} \leq \1{H'}
\sum_{j' \in J'} \delta_{Y_{j'}} \right]
\geq 1 - |H'| \, \exp \left\{ - (\rho' - \rho)L
+ \left(\frac{c \rho L^2 (\log n)^d}{\sqrt{n}} \right) \right\}
\end{equation}
for all finite $H' \subset \Z^d$ and all $n \geq c L^2$.
\end{lemma}

\begin{proof}
By Lemma~\ref{l:couplesystem}, there exists a coupling $\mathbb{Q}$ such that
\begin{equation}
\label{e:coupleGG}
\mathbb{Q} \left[ \1{H'} \sum_{j \in J} \delta_{Z^j_n} \leq \1{H'}  \sum_{j' \in J'} \delta_{Y_{j'}} \right]
\geq \mathbb{Q} \big[ G_J(z) \leq \rho'\,\,\forall\,z \in H'\big],
\end{equation}
where $G_J(z) = \sum_{j \in J} \xi_j \, p_{n}(x_j, z)$ with $(\xi_j)_j$ i.i.d.\ Exp$(1)$
random variables. 
Write
\begin{align}
\nonumber
&\mathbb{Q} \big[\exists\, z \in H'\colon\,G_J(z) > \rho'\big]
\leq |H'| \sup_{z \in H'}\mathbb{Q} [ G_J(z) > \rho']\\
& \label{e:QGJz}
\qquad \leq |H'| \, e^{-\rho'L} \sup_{z \in H'}E^{\mathbb{Q}} \big[ \exp\{ L G_J(z) \} \big].
\end{align}
If $n \geq c L^2$ with large enough $c \ge 1$, then, by \eqref{e:localclt},
\begin{equation}
\label{e:localclt2}
\sup_{x \in \mathbb{Z}} L p_n(0,x) \leq \frac12.
\end{equation}
Thus we may write, for any $z \in \mathbb{Z}$,
\begin{equation}
E^{\mathbb{Q}} \big[ \exp\{ L G_J(z) \} \big] = \smash{\prod_{j \in J}}
E^{\mathbb{Q}} \big[ \exp\{\xi_j L p_n(x_j,z) \} \big]
= \prod_{\smash{j \in J}} \Big(1 - L p_n(x_j,z) \Big)^{-1}.
\end{equation}
Using \eqref{e:localclt2} and $-\log(1-x) \le x+x^2$ for all $x \in [0,1/2]$,
we obtain
\begin{equation}
\begin{array}{e}
\prod_{j \in J}(1 - L p_n(z,x_j))^{-1}
& \leq & \prod_{j \in J} \exp \big\{ L p_n(z,x_j) \big(1 + L p_n(z,x_j) \big) \big\}\\
& \leq & \exp \Big\{\sum_{j \in J} L p_n(z,x_j) \big( 1 + \sup_{x \in \mathbb{Z}} L p_n (0,x)  \big) \Big\}\\
& \leq & \exp \Big\{ \rho L \big(1 + \tfrac{cL (\log n)^d}{\sqrt{n}} \big) \big( 1 + \tfrac{cL}{\sqrt{n}}  \big) \Big\}\\
&\leq & \exp \Big\{\rho L \big(1 + \tfrac{c' L (\log n)^d}{\sqrt{n}} \big) \Big\}.
\end{array}
\end{equation}
where the last two inequalities are justified using $n \ge c L^2 \ge L$, Lemma~\ref{l:integration} and \eqref{e:localclt}.
Inserting this estimate into \eqref{e:QGJz}, we get the claim.
\end{proof}

\subsection{Proof of Theorem~\ref{t:decouple}}
\label{ss:proofthmdecouple}
We can now finish the:
\begin{proof}[Proof of Theorem~\ref{t:decouple}]
The proof of item (a) can be obtained by adapting Appendixes B--C of \cite{HHSST14} to higher dimensions as follows.
First of all, (B.1)--(B.3) therein should be substituted by their $d$-dimensional counterparts \eqref{e:localclt}--\eqref{e:SRW2} above.
Definition~B.1 therein should be substituted by the $\rho$-dense analogue of Definition~\ref{d:balanced} above, i.e.,
changing ``$\le$'' to ``$\ge$'' in \eqref{e:rhodense}.
One may then follow the arguments given in \cite{HHSST14} to re-obtain
Lemmas~B.2--B.3 therein with the following differences:
in both (B.6) and (B.10) therein, $\log n$ should be substituted by $(\log n)^d$ (analogously to Lemmas~\ref{l:integration}--\ref{l:couple} above).
The proof of Theorem~\ref{t:decouple}(a) then follows 
from these results exactly as in the proof of Theorem~C.1 in \cite{HHSST14}.
The proof of Theorem~\ref{t:decouple}(b) is completely analogous, following from Lemma~\ref{l:couple} above
as Theorem~C.1 in \cite{HHSST14} follows from Lemma~B.3 therein.
\end{proof}

%%%%%% End of text for appendixmany%%%%

\bibliographystyle{plain}
\bibliography{all}

\end{document}